\documentclass[11pt]{amsart}
\usepackage[margin=1.1in]{geometry}
\usepackage[T1]{fontenc}
\usepackage{lmodern}
\usepackage{amsmath,amssymb,mathtools,mathrsfs}
\usepackage{microtype}
\usepackage{needspace}
\usepackage{tikz}
\usetikzlibrary{arrows.meta,positioning}
\usepackage{tikz-cd}
\usepackage[hidelinks]{hyperref}

\hypersetup{
	bookmarksdepth=3,
	colorlinks,
	linkcolor={red!50!black},
	citecolor={blue!50!black},
	urlcolor={blue!80!black}
}

\numberwithin{equation}{section}
\newtheorem{theorem}{Theorem}[section]
\newtheorem{proposition}[theorem]{Proposition}
\newtheorem{lemma}[theorem]{Lemma}
\newtheorem{corollary}[theorem]{Corollary}
\theoremstyle{definition}

\theoremstyle{remark}
\newtheorem{remark}[theorem]{Remark}
\newtheorem{example}[theorem]{Example}

\newcommand{\op}[1]{\operatorname{#1}}

\newcommand{\E}{\mathcal E}
\newcommand{\C}{\mathcal C}
\newcommand{\Pcal}{\mathcal P}
\newcommand{\CC}{\mathrm{CC}}
\newcommand{\kk}{\mathbb C}
\newcommand{\PP}{\mathbb P}
\newcommand{\ZZ}{\mathbb Z}
\newcommand{\add}{\op{add}}
\newcommand{\Hom}{\op{Hom}}
\newcommand{\Ext}{\op{Ext}}
\newcommand{\End}{\op{End}}
\newcommand{\Gr}{\op{Gr}}
\newcommand{\CM}{\op{CM}}
\newcommand{\GP}{\op{GP}}
\newcommand{\ind}{\op{ind}}
\newcommand{\coker}{\op{coker}}
\newcommand{\Frac}{\op{Frac}}
\newcommand{\tw}{\op{tw}}
\newcommand{\eps}{\varepsilon}
\newcommand{\RR}{\mathbb R}
\newcommand{\QQ}{\mathbb Q}
\newcommand{\Newt}{\op{Newt}}
\newcommand{\NP}{\mathsf N}
\newcommand{\dv}{\underline{\dim}}
\newcommand{\coind}{\op{coind}}
\newcommand{\res}{\op{res}}

\title{Syzygy Transformations of Cluster Characters and Geometric Twists}
\author{Jiarui Fei}
\address{School of Mathematical Sciences, Shanghai Jiao Tong University}
\email{jiarui@sjtu.edu.cn}
\thanks{The author was supported in part by the National Natural Science Foundation of China (Nos.~12131015 and 12571038).}
\date{}
\hypersetup{pdftitle={Syzygy Transformations of Cluster Characters and Geometric Twists},pdfsubject={Syzygies, geometric twists, and tropical applications to all cluster characters}}
\begin{document}

\begin{abstract}
We use the full multiplication formula to study syzygy and suspension of
cluster characters. In stably 2-Calabi--Yau Frobenius categories, syzygy
and cosyzygy induce inverse automorphisms of localized character algebras;
in Hom-finite 2-Calabi--Yau triangulated categories, the same uniqueness
principle transports all cluster characters under suspension. The latter
gives our Auslander--Reiten $F$-polynomial identity. No mutation
reachability is required. We extend positroid twist identities to all
objects, recover the Grassmannian and unipotent-cell formulas, identify
partition-function algebras before boundary localization, and derive
tropical covariance when the relevant tropical coordinate map is a
homeomorphism.
\end{abstract}
\maketitle

\section*{Introduction}

We study the action of \emph{syzygy} on \emph{cluster characters}.
Let $\E$ be a \emph{Frobenius category} whose \emph{stable category}
is Hom-finite and 2-Calabi--Yau. Given a cluster character $\CC$, choose a \emph{conflation}
\[ 0\longrightarrow\Omega M\longrightarrow P_M\longrightarrow M\longrightarrow0 \]
with $P_M$ \emph{projective-injective}. We consider the character value $\CC(\Omega M)/\CC(P_M)$.
The denominator records the projective-injective term of the conflation. In Theorem~\ref{thm:main}, we prove
that these values respect all algebraic relations among the original
characters and define an automorphism of the \emph{localized character
algebra}. The inverse is given by the corresponding character value of
a \emph{cosyzygy}, defined by a conflation
$0\to M\to I_M\to\Omega^{-1}M\to0$ with $I_M$ projective-injective:
\[ \mathsf t_-\bigl(\CC(M)\bigr)=\frac{\CC(\Omega M)}{\CC(P_M)},\qquad \mathsf t_+\bigl(\CC(M)\bigr)=\frac{\CC(\Omega^{-1}M)}{\CC(I_M)}. \]
We call $\mathsf t_-$ and $\mathsf t_+$ the \emph{syzygy and
cosyzygy transformations} of cluster characters. We assume that the characters of the syzygies of the initial summands
are nonzero.
There is no reachability, rigidity, or genericity condition on $M$.
For each \emph{frozen variable} $x_f=\CC(T_f)$, where $T_f$ is
projective-injective, both transformations send $x_f$ to its
reciprocal $x_f^{-1}$.

There is a parallel triangulated statement. For a Hom-finite
2-Calabi--Yau triangulated category with cluster-tilting object
$U=\bigoplus_iU_i$, Theorem~\ref{thm:triangulated-transport} says that a
coordinate homomorphism agreeing with suspension on the characters of
$\Sigma U_i$ transports every cluster character under $\Sigma$. This is
an intrinsic transport statement, independent of a Frobenius realization;
birationality is a separate issue.

Both results come from the established \emph{full multiplication formula}.
Palu proves the triangulated formula \cite[Theorem~1.1]{Pal12}, while the
Frobenius forms used here are recalled in
Sections~\ref{sec:frobenius} and \ref{sec:geometric}. A multiplicative
function satisfying the full formula is determined by its values on one
cluster-tilting object. In the Frobenius setting the horseshoe lemma
applies this principle to $\CC(\Omega M)/\CC(P_M)$; in the triangulated
setting suspension preserves the formula. We use the multiplication-formula
method of Assem--Dupont--Schiffler \cite[Theorem~6.13]{ADS14} and
Keller--Plamondon--Qin \cite[Theorem~4.1 and Remark~4.2]{KPQ}.

Our original motivation was the \emph{Auslander--Reiten transformation
formula} for $F$-polynomials announced in \cite[Theorem~1.9]{Fei}.
We prove this identity as Theorem~\ref{thm:AR} by applying the
triangulated transport theorem to the generalized cluster category.
\subsection*{Geometric twists}
We identify a geometric twist by comparing its values on one cluster;
Theorem~\ref{thm:main} then gives its action on every character. For
Grassmannians and unipotent cells, the resulting all-object formulas were
already known: we recover the Grassmannian formula of Jensen--King--Su
\cite[Proposition~9.14]{JKS24} and the unipotent-cell formula of
Geiss--Leclerc--Schr\"oer \cite[Theorem~6]{GLS12}.

The new geometric application is the positroid case. \c{C}anak\c{c}\i--King--Pressland
proved the rank-one formula \cite[Theorem~12.2]{CKP24}, and Pressland
extended it to reachable rigid objects \cite[Theorem~5.27]{Pre}. We prove
the formula for every object of $\GP\mathcal B_\Pcal$, for every positroid
$\Pcal$, including disconnected ones. In particular, the normalized
partition function of Jensen--Riordan--Su is the twist of the character on
every object of this category, answering \cite[Remark~7.7]{JRS26} in this
setting. Corollary~\ref{cor:partition-algebra} shows that these partition
functions generate the cluster algebra of the partition seed before
boundary localization, removing the reachability condition in
\cite[Remark~7.13]{JRS26}. We also identify the seed isomorphism with the
geometric twist after projection and determine its value on every character.
The disconnected case is reduced to the connected one by comparing the
frozen factorization ideals of the relevant corner algebras.

The same cosyzygy character values appear in Fraser--Keller's
construction of a \emph{quasi-cluster automorphism}, under reachability
of the shifted cluster-tilting object
\cite[Theorem~A.9 and Example~A.10]{FrK23}. We use the multiplication
formula to dispense with this reachability condition and obtain the
identity for all objects. Our conclusion concerns the function field
and the localized character algebra; the quasi-cluster property is a
separate assertion. The horseshoe argument extends Pressland's
exchange-sequence calculation \cite[proof of Theorem~5.27]{Pre} to
extension spaces of arbitrary dimension.

\subsection*{AR-translation of \texorpdfstring{$F$}{F}-polynomials and tropical applications}
The motivating identity, Theorem~\ref{thm:AR}, takes the form
\[ F_M(\psi(z))=z^{-\dv M}F_I(z)^{\check\delta_M}F_{\tau M}(z),\qquad \psi_i(z)=z_i^{-1}\prod_jF_{I_j}(z)^{b_{ij}}. \]
Here $B=(b_{ij})$ is the exchange matrix, $\check\delta_M$ is the
coweight of $M$, and $F_I(z)^c=\prod_iF_{I_i}(z)^{c_i}$. The variables
$z_1,\ldots,z_m$ are algebraically independent, one for each quiver
vertex. Theorem~\ref{thm:AR} proves the identity in
$\QQ(z_1,\ldots,z_m)$, with no rank assumption on $B$. The use of
independent $F$-variables is essential when $B$ is singular, since the
usual monomial specialization can lose information.

The proof applies Theorem~\ref{thm:triangulated-transport} in the
generalized cluster category, where the characters of $\Sigma^2T_i$ are
given by injective $F$-polynomials. Proposition~\ref{prop:AR-inverse}
then gives the birational inverse in terms of the dual $F$-polynomials of
projectives, and Corollary~\ref{cor:triangulated-character} gives the
corresponding suspension automorphism.

We write $\CC^{\mathrm{trop}}_M$ for the \emph{tropical cluster
character}, and reserve $f_M$ for the tropical $F$-polynomial,
following \cite{FeiT23}. Let $D_-$ and $D_+$ be the tropical coordinate
maps of the character transformation and its inverse. If either map is a
homeomorphism, then they are inverse integral piecewise-linear
homeomorphisms. If $\CC(P_M)=x^{p_M}$, then
\[ \CC^{\mathrm{trop}}_{\Omega M}(\omega)=\CC^{\mathrm{trop}}_M(D_-(\omega))+\omega(p_M). \]
No positivity of the character coefficients is required: the homeomorphism
condition prevents cancellation from changing the leading degree.
Section~\ref{subsec:tropical-AR} gives the parallel transformation laws for
tropical $F$-polynomials and Hom dimensions, followed by their pairing
identities.

The \emph{Markov quiver} provides a running example in which the initial
cluster-tilting object and its shift are not mutation-connected
\cite[Example~4.3]{Pla13}. We compute the projective $F$-polynomials
explicitly, obtain the suspension transformation on all cluster characters,
and verify the Auslander--Reiten identity on a nonrigid family.

We also relate the AR substitution to stability scattering. With principal
coefficients, the injective and projective $F$-polynomials determine the
chamber products and yield a criterion for comparison with the cluster
scattering diagram.

\subsection*{Organization}
Section~\ref{sec:frobenius} proves the Frobenius syzygy theorem and the
parallel triangulated suspension-transport theorem. Section~\ref{sec:AR}
applies the latter to Auslander--Reiten translation and $F$-polynomials.
We identify the geometric twists and their normalizations in
Section~\ref{sec:geometric}. Section~\ref{sec:tropical} contains the
tropical results and examples. Appendix~\ref{app:positroid} gives the
convention comparison and the frozen-ideal argument for disconnected
positroids.

\section{Syzygy, cosyzygy, and suspension of cluster characters}\label{sec:frobenius}

We first prove the Frobenius statement and then its intrinsic triangulated
counterpart, used in Section~\ref{sec:AR}.

We work over $\mathbb C$, and write $\chi_c$ for compactly supported
Euler characteristic. Let $\E$ be an idempotent-complete,
Krull--Schmidt, $\kk$-linear Frobenius exact category. Suppose that
$\Ext^1_\E(M,N)$ is finite-dimensional for every pair of objects, and
that the stable category
$\C=\underline\E$
is Hom-finite and 2-Calabi--Yau. In this quotient, morphisms factoring
through projective-injective objects are set to zero. Its suspension is denoted by $\Sigma$.
No Hom-finiteness assumption is imposed on $\E$ itself.

We fix a basic cluster-tilting object
$T=T_1\oplus\cdots\oplus T_n$.
In particular, $\Ext^1_\E(T,M)=0$ if and only if
$M\in\add T$. Let $\mathsf F\subseteq\{1,\ldots,n\}$ be the set of indices
of its projective-injective summands. These include every indecomposable
projective-injective object of $\E$. The variables $x_f$, $f\in \mathsf F$,
are the frozen variables.

Put
\[ R=\kk[x_1^{\pm1},\ldots,x_n^{\pm1}],\qquad K=\Frac R. \]
Let $\CC:\op{obj}\E\to R$ be isomorphism-invariant, with
\[ \CC(0)=1,\qquad \CC(M\oplus N)=\CC(M)\CC(N),\qquad \CC(T_i)=x_i. \]
We set
\[ \mathscr A_\CC=\kk[\,\CC(M):M\in\E\,][x_f^{-1}:f\in \mathsf F]. \]
We use $R$ for character expansions and $K$ for coordinate
substitutions; the transformations will preserve $\mathscr A_\CC$,
not necessarily the fixed Laurent ring $R$.
We work with a cluster character to which the full multiplication theorem
applies. We use its Euler-integrated form: for all $M,N$,
\begin{equation}\label{eq:full}
\begin{aligned}
 &\dim\Ext^1_\E(M,N)\,\CC(M)\CC(N)\\
 =&\int_{\PP\Ext^1_\E(M,N)}\CC(E_\eps)\,d\chi_c
       +\int_{\PP\Ext^1_\E(N,M)}\CC(E_\eta)\,d\chi_c.
\end{aligned}
\end{equation}
For $\eps\in\Ext^1_\E(M,N)$, we write $E_\eps$ for the middle
term of a conflation $0\to N\to E_\eps\to M\to0$. As part of the
multiplication formula, both integrands in \eqref{eq:full} are finite-valued
constructible functions.
Thus each integral is a finite sum of values weighted by Euler
characteristics. If the extension space is zero, both sides are zero.

The multiplication theorem supplies \eqref{eq:full} and the
constructibility used below. Palu proves
the triangulated version in \cite[Theorem~1.1]{Pal12}; the refined
Hom-finite Frobenius version is \cite[Theorem~3.6]{KPQ}. For the
Hom-infinite Grassmannian and positroid categories, we use the
preprojective formula and its Grassmannian homogenization, as explained
in Section~\ref{sec:geometric}. We state the abstract argument for any
character satisfying this formula, including its constructibility
assertion.

For each $M$, we choose conflations
\begin{equation}\label{eq:presentations} 0\longrightarrow\Omega M\longrightarrow P_M\longrightarrow M \longrightarrow0, \qquad 0\longrightarrow M\longrightarrow I_M\longrightarrow\Omega^{-1}M \longrightarrow0,\end{equation}
where $P_M,I_M$ are projective-injective. The \emph{projective deflation} and \emph{injective inflation} need not be minimal. We associate to them the character values
\begin{equation}\label{eq:normalized} \CC^-(M)=\frac{\CC(\Omega M)}{\CC(P_M)},\qquad \CC^+(M)=\frac{\CC(\Omega^{-1}M)}{\CC(I_M)}.\end{equation}
The division takes place in the character ring: $\CC(P_M)$ and
$\CC(I_M)$ are frozen monomials, so both expressions belong to $R$.
Changing the chosen presentation leaves these ratios unchanged, although
it may add projective-injective summands to the syzygy or cosyzygy.

\begin{theorem}[Syzygy and cosyzygy on cluster characters]\label{thm:main}
Assume that
\begin{equation}\label{eq:nonvanishing} \CC(\Omega T_i)\ne0\qquad(1\leq i\leq n).\end{equation}
The assignments
\[ \mathsf t_-(x_i)=\CC^-(T_i),\qquad \mathsf t_+(x_i)=\CC^+(T_i) \]
extend uniquely to inverse $\kk$-automorphisms of $K$. For every
object $M\in\E$,
\begin{equation}\label{eq:objectwise} \mathsf t_-\bigl(\CC(M)\bigr)=\frac{\CC(\Omega M)}{\CC(P_M)}, \qquad \mathsf t_+\bigl(\CC(M)\bigr)= \frac{\CC(\Omega^{-1}M)}{\CC(I_M)}.\end{equation}
Both maps restrict to inverse automorphisms of $\mathscr A_\CC$,
and at every frozen vertex,
\[ \mathsf t_-(x_f)=\mathsf t_+(x_f)=x_f^{-1}\qquad(f\in \mathsf F). \]

\end{theorem}

For a projective-injective summand $T_f$, we use its identity map for
both presentations. Then $\Omega T_f=\Omega^{-1}T_f=0$ and
$P_{T_f}=I_{T_f}=T_f$. Since $\CC(0)=1$, both character ratios
are $1/x_f$.

\begin{corollary}\label{cor:seed}
Let $\Theta$ be a $\kk$-automorphism of $K$. If
\[ \Theta(x_i)=\frac{\CC(\Omega T_i)}{\CC(P_{T_i})} \qquad(1\leq i\leq n), \]
then, for every $M\in\E$,
\[ \Theta\bigl(\CC(M)\bigr)=\frac{\CC(\Omega M)}{\CC(P_M)},\qquad \Theta^{-1}\bigl(\CC(M)\bigr)= \frac{\CC(\Omega^{-1}M)}{\CC(I_M)}. \]
\end{corollary}

In Corollary~\ref{cor:seed}, the nonvanishing condition follows because
$\Theta$ is an automorphism. Neither statement requires the syzygies
of the summands of $T$ to be reachable or an exchange matrix to have
full rank.

\subsection{Uniqueness from the multiplication formula}\label{sec:uniqueness}

\begin{lemma}[Frobenius uniqueness]\label{lem:uniqueness}
Let $\Xi_1,\Xi_2:\op{obj}\E\to\mathbb F$ be isomorphism-invariant functions
with values in a field $\mathbb F$ of characteristic zero. Suppose that each
is multiplicative, takes the value $1$ at zero, and satisfies
\eqref{eq:full}, with finite-valued constructible integrands. If
$\Xi_1(T_i)=\Xi_2(T_i)\ne0$ for $1\leq i\leq n$, 
then $\Xi_1(M)=\Xi_2(M)$ for every $M\in\E$.
\end{lemma}

\begin{proof}
We write $\bar M$ for the image of $M$ in $\C$, and induct on
\[ \ell(M)=\dim_\kk\Ext^1_\E(T,M) =\dim_\kk\C(\bar T,\Sigma\bar M). \]
If $\ell(M)=0$, then $M\in\add T$, so multiplicativity proves the
assertion, including when $M$ has projective-injective summands.

Suppose $\ell(M)>0$. We choose $i$ such that
$r_i=\dim\Ext^1_\E(T_i,M)>0$.
2-Calabi--Yau duality gives
$\dim\Ext^1_\E(M,T_i)=r_i$. We show that every middle term in either
projectivized extension space has smaller $\ell$.

First, let $0\ne a:\bar T_i\to\Sigma\bar M$ represent an extension
with triangle
\[ \bar M\longrightarrow\bar E_a\longrightarrow\bar T_i \xrightarrow{a}\Sigma\bar M. \]
Applying $\C(\bar T,-)$ and using
$\C(\bar T,\Sigma\bar T_i)=0$ gives
\[ \C(\bar T,\Sigma\bar E_a) \simeq\coker\bigl(\C(\bar T,\bar T_i) \xrightarrow{a_*}\C(\bar T,\Sigma\bar M)\bigr). \]
The map $a_*$ is nonzero, since evaluation on the summand projection
$\bar T\to\bar T_i$ detects $a$. Hence $\ell(E_a)<\ell(M)$.

Second, let $0\ne b:\bar M\to\Sigma\bar T_i$ represent a triangle
\[ \bar T_i\longrightarrow\bar E_b\longrightarrow\bar M \xrightarrow{b}\Sigma\bar T_i. \]
We suspend the triangle and apply $\C(\bar T,-)$. We obtain
\[ \C(\bar T,\Sigma\bar E_b) \simeq\ker\bigl(\C(\bar T,\Sigma\bar M) \xrightarrow{(\Sigma b)_*}\C(\bar T,\Sigma^2\bar T_i)\bigr). \]
By the nondegenerate 2-Calabi--Yau pairing, there is a morphism
$a:\bar T_i\to\Sigma\bar M$ such that $(\Sigma b)a\ne0$.
Precomposing $a$ with the summand projection $\bar T\to\bar T_i$
shows that $(\Sigma b)_*\ne0$. Hence $\ell(E_b)<\ell(M)$.

Applying the full multiplication formula to $(T_i,M)$, for either
$\Xi=\Xi_1$ or $\Xi=\Xi_2$, gives
\[ r_i \Xi(T_i)\Xi(M) =\int_{\PP\Ext^1(T_i,M)}\Xi(E_a)\,d\chi_c +\int_{\PP\Ext^1(M,T_i)}\Xi(E_b)\,d\chi_c. \]
By induction the two right-hand sides agree. We divide by the common
nonzero coefficient $r_i \Xi_1(T_i)=r_i \Xi_2(T_i)$ to complete the proof.
\end{proof}

Only the pairs $(T_i,M)$ are needed in this induction. Their extension
spaces may have any positive dimension.

We also need the corresponding statement in a triangulated category.
It is independent of a Frobenius realization. For a function on a
triangulated category, the \emph{full constructible multiplication formula}
means the triangulated analogue of \eqref{eq:full}, with
$\Ext^1_{\mathcal T}(X,Y)=\mathcal T(X,\Sigma Y)$ and middle terms of
triangles.

\begin{lemma}[Triangulated uniqueness]\label{lem:triangulated-uniqueness}
Let $\mathcal T$ be a Hom-finite, $\kk$-linear 2-Calabi--Yau
triangulated category with a basic cluster-tilting object $U=\bigoplus_{i=1}^mU_i$. Let
$\Xi_1,\Xi_2:\op{obj}\mathcal T\to\mathbb F$ be isomorphism-invariant
functions with values in a field $\mathbb F$ of characteristic zero.
Suppose that both functions are multiplicative, take the value $1$ at
zero, and satisfy the full constructible multiplication formula. If
$\Xi_1(\Sigma U_i)=\Xi_2(\Sigma U_i)\ne0$ for $1\leq i\leq m$,
then $\Xi_1(L)=\Xi_2(L)$ for every $L\in\mathcal T$.
\end{lemma}

\begin{proof}
Put $H=\mathcal T(U,-)$ and induct on
$\ell(L)=\dim_\kk H(L)$. If $\ell(L)=0$, then
$L\in\add\Sigma U$. Otherwise choose $i$ with
$\mathcal T(U_i,L)\ne0$. For the two extension families in the
multiplication formula for $(L,\Sigma U_i)$, take triangles
\[
U_i\xrightarrow{a}L\longrightarrow E_a\longrightarrow\Sigma U_i,
\qquad
\Sigma U_i\longrightarrow E_b\longrightarrow L\xrightarrow{b}\Sigma^2U_i.
\]
Applying $H$ gives
\[
H(E_a)=\coker H(a),\qquad H(E_b)=\ker H(b).
\]
For $a\ne0$, the summand projection $U\to U_i$ shows $H(a)\ne0$.
For $b\ne0$, the 2-Calabi--Yau pairing gives $a:U_i\to L$ with
$ba\ne0$, hence $H(b)\ne0$. Thus both middle terms have smaller
$\ell$. The two extension spaces have the same positive dimension
$r_i=\dim_\kk\mathcal T(U_i,L)$ by 2-Calabi--Yau duality. The induction
hypothesis identifies both integrals, and division by the common nonzero
coefficient $r_i\Xi_1(\Sigma U_i)=r_i\Xi_2(\Sigma U_i)$ completes the
proof.
\end{proof}

\begin{theorem}[Suspension for cluster characters]
\label{thm:triangulated-transport}
Let $\mathcal T$ and $U$ be as in Lemma~\ref{lem:triangulated-uniqueness},
and let $\CC_U(-;x)$ be a cluster character with values in
$\mathbb K[x_1^{\pm1},\ldots,x_m^{\pm1}]$, where $\mathbb K$ has
characteristic zero, satisfying the full constructible multiplication
formula and normalized by $\CC_U(\Sigma U_i;x)=x_i$. Let $\mathbb F$
be an extension field of $\mathbb K$, let
$\xi_1,\ldots,\xi_m\in\mathbb F^\times$, and write $\CC_U(L;\xi)$ for
evaluation at $x_i=\xi_i$. If a $\mathbb K$-algebra homomorphism
$\phi:\mathbb K[x_1^{\pm1},\ldots,x_m^{\pm1}]\to\mathbb F$ satisfies
\[ \phi(x_i)=\CC_U(\Sigma^2U_i;\xi) \qquad(1\leq i\leq m), \]
then, for every $L\in\mathcal T$,
\[ \phi\bigl(\CC_U(L;x)\bigr)=\CC_U(\Sigma L;\xi). \]
\end{theorem}

\begin{proof}
The two sides define multiplicative functions of $L$ satisfying the full
constructible multiplication formula: apply $\phi$ to its finite
constructible sums on the left, and suspend the triangles on the right.
They agree with the same nonzero value on every $\Sigma U_i$, so
Lemma~\ref{lem:triangulated-uniqueness} applies.
\end{proof}

Birationality is not part of Theorem~\ref{thm:triangulated-transport};
in Section~\ref{sec:AR} it follows from the projective and injective
$F$-polynomial identities.

\subsection{Characters of syzygies and constructible integration}\label{sec:full-syzygy}

\begin{lemma}\label{lem:schange}
The two expressions in \eqref{eq:normalized} are independent of the
chosen conflations. They are multiplicative, take the value $1$ at
zero, and satisfy
\[ \CC^-(P)=\CC^+(P)=\CC(P)^{-1} \]
for every projective-injective object $P$.
\end{lemma}

\begin{proof}
We compare two projective deflations $P_1\to M$, $P_2\to M$, with
kernels $K_1,K_2$. By taking their pullback we obtain Schanuel's
isomorphism $K_1\oplus P_2\simeq K_2\oplus P_1$.
We apply multiplicativity and divide by $\CC(P_1)\CC(P_2)$ to prove
independence for $\CC^-$. The dual argument proves independence for
$\CC^+$. We prove multiplicativity by taking direct sums of
conflations, and the last assertion by using $P\xrightarrow{1}P$.
\end{proof}

We now extend the horseshoe calculation of
\cite[proof of Theorem~5.27]{Pre} to the full multiplication formula.

\begin{lemma}\label{lem:horseshoe}
The functions $\CC^-$ and $\CC^+$ satisfy the full multiplication
formula \eqref{eq:full}, with finite-valued constructible integrands.
\end{lemma}

\begin{proof}
We fix projective deflations $P_M\to M$ and $P_N\to N$. Passing to
the stable category $\C=\underline\E$, the syzygy functor is inverse to
suspension. Hence the autoequivalence $\Sigma^{-1}$, together with the
identifications $\overline{\Omega M}\simeq\Sigma^{-1}\bar M$ and
$\overline{\Omega N}\simeq\Sigma^{-1}\bar N$, gives a linear
isomorphism
\begin{equation}\label{eq:ext-isomorphism}
s_{M,N}:\Ext^1_\E(M,N) \xrightarrow{\sim}
\Ext^1_\E(\Omega M,\Omega N).
\end{equation}
The horseshoe construction below realizes this stable-category
isomorphism on extension classes. Thus \eqref{eq:ext-isomorphism} is
fixed without making algebraic choices of extension representatives in
families.

For an extension $0\to N\to E\to M\to0$, the horseshoe construction
gives a diagram of conflations
\[
\begin{tikzcd}[column sep=large,row sep=large]
0\arrow[r]&\Omega N\arrow[r]\arrow[d]&K_E\arrow[r]\arrow[d]&\Omega M\arrow[r]\arrow[d]&0\\
0\arrow[r]&P_N\arrow[r]\arrow[d]&P_N\oplus P_M\arrow[r]\arrow[d]&P_M\arrow[r]\arrow[d]&0\\
0\arrow[r]&N\arrow[r]&E\arrow[r]&M\arrow[r]&0.
\end{tikzcd}
\]
The upper extension represents $s_{M,N}(\eps)$, up to the immaterial
sign arising from the convention for suspension. The middle column is
a conflation with projective-injective middle term. Consequently,
\begin{equation}\label{eq:middle-value} \CC(K_E)=\CC(P_M)\CC(P_N)\CC^-(E).\end{equation}
The same identity holds for extensions in the reverse direction, with
the same denominator $\CC(P_M)\CC(P_N)$.

We apply \eqref{eq:full} to $(\Omega M,\Omega N)$, pull back the two
integrals along the projectivizations of $s_{M,N}$ and $s_{N,M}$,
and use \eqref{eq:middle-value}. We then divide by
$\CC(P_M)\CC(P_N)$ to obtain \eqref{eq:full} for $\CC^-$.

This argument also proves the required constructibility. The function
on the original extension space is the pullback of a finite-valued
constructible function under a linear isomorphism, divided by a fixed
nonzero scalar in $K$. No algebraic choice of minimal projective covers
of varying middle terms is required. The dual horseshoe construction
proves the assertion for $\CC^+$.
\end{proof}

We keep any projective-injective summands of $K_E$ in
\eqref{eq:middle-value}. Such summands can occur even when the end
presentations are minimal, and their characters contribute frozen
factors to the formula.

\subsection{Birationality and the inverse}\label{sec:birational}

\begin{proof}[Proof of Theorem~\ref{thm:main}]
Put $q_i=\CC^-(T_i)$. By
\eqref{eq:nonvanishing}, the assignments $x_i\mapsto q_i$ define a
homomorphism $\phi:R\to K$.
We will prove that $\phi$ extends to a field automorphism.

The functions $\phi\circ\CC$ and $\CC^-$ satisfy the full
multiplication formula. For the first, apply $\phi$ to the finite
constructible sums for $\CC$; its new fibres are finite unions of the
old strata. For the second, use Lemma~\ref{lem:horseshoe}. They agree on
all $T_i$, with nonzero common values. Lemma~\ref{lem:uniqueness} yields
\begin{equation}\label{eq:phi-all} \phi(\CC(M))=\CC^-(M)\qquad(M\in\E).\end{equation}

We choose $0\to T_i\to I_{T_i}\to C_i\to0$, with $I_{T_i}$ projective-injective,
and put $r_i=\frac{\CC(C_i)}{\CC(I_{T_i})}\in R$.
The same conflation gives a projective deflation onto $C_i$, with kernel
$T_i$. Thus \eqref{eq:phi-all} and Lemma~\ref{lem:schange} give
\[ \phi(\CC(C_i))=\frac{x_i}{\CC(I_{T_i})},\qquad \phi(\CC(I_{T_i}))=\CC(I_{T_i})^{-1},\qquad \phi(r_i)=x_i. \]
In particular, $\kk(q_1,\ldots,q_n)$ contains every $x_i$. It is
therefore $K$, and the $q_i$ are algebraically independent. Hence
$\phi$ extends to an automorphism $\mathsf t_-$ of $K$, whose
inverse sends $x_i$ to $r_i=\CC^+(T_i)$.

For an arbitrary second conflation in \eqref{eq:presentations}, applying
\eqref{eq:phi-all} once more gives
\[ \mathsf t_-\left(\frac{\CC(\Omega^{-1}M)}{\CC(I_M)}\right)=\CC(M). \]
This proves the second identity in \eqref{eq:objectwise}; the first is
\eqref{eq:phi-all}. Both maps send $x_f$ to $x_f^{-1}$ by
Lemma~\ref{lem:schange}. Finally, every expression on the right of
\eqref{eq:objectwise} belongs to $\mathscr A_\CC$. Thus both maps
preserve this algebra and are inverse automorphisms of it.
\end{proof}

\begin{proof}[Proof of Corollary~\ref{cor:seed}]
The specified values are nonzero because $\Theta$ is a field automorphism. Theorem~\ref{thm:main} applies, and the two field
automorphisms agree on a transcendence basis.
\end{proof}

\begin{remark}\label{rem:scope}
The automorphisms preserve all algebraic relations among character values,
with each frozen variable replaced by its reciprocal. Their construction takes place in
$\mathscr A_\CC$; it uses neither a basis theorem nor an identification
with an upper cluster algebra. The syzygy functor acts on the stable
category, and the projective factors make the character formulas
independent of its chosen lifts to $\E$.
\end{remark}
We choose projective deflations
successively and write $M_j=\Omega^j M$, with $M_0=M$, and
$0\to M_{j+1}\to P_j\to M_j\to0$. Then, for $r\geq1$,
\[ \mathsf t_-^r\bigl(\CC(M)\bigr) =\CC(M_r)\prod_{j=0}^{r-1}\CC(P_j)^{(-1)^{r-j}}. \]
We prove this by induction, using
$\mathsf t_-(\CC(P_j))=\CC(P_j)^{-1}$ at each step.
The expression is independent of the chosen conflations because it
equals the left-hand side. Periodicity in the stable category need not
give periodicity of the character transformation: the projective factors
may not cancel.

\subsection{The Markov quiver: nonreachability}\label{subsec:markov-setup}
Consider the Markov quiver 
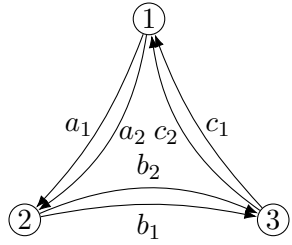
\begin{center}
\begin{tikzpicture}[>=Latex, every node/.style={circle, draw, inner sep=1.1pt, minimum size=11pt},
  lbl/.style={draw=none, circle=none, fill=none, inner sep=1pt}]
  \node (1) at (90:1.7) {$1$};
  \node (2) at (210:1.9) {$2$};
  \node (3) at (330:1.9) {$3$};
  \draw[->, bend left=11] (1) to node[lbl, midway, left] {$a_1$} (2);
  \draw[->, bend left=23] (1) to node[lbl, midway, right] {$a_2$} (2);
  \draw[->, bend left=11] (2) to node[lbl, midway, below] {$b_1$} (3);
  \draw[->, bend left=23] (2) to node[lbl, midway, above] {$b_2$} (3);
  \draw[->, bend left=11] (3) to node[lbl, midway, right] {$c_1$} (1);
  \draw[->, bend left=23] (3) to node[lbl, midway, left] {$c_2$} (1);
\end{tikzpicture}
\end{center}
We multiply paths in traversal order and use the potential
\begin{equation}\label{eq:markov-potential} W=a_1b_1c_1+a_2b_2c_2-a_1b_2c_1a_2b_1c_2.\end{equation}
This is the Labardini potential in \cite[Example~4.3]{Pla13}, with path
multiplication reversed. Its completed Jacobian algebra $J$ is finite-dimensional and the potential
is nondegenerate. We use $W$ throughout the example.

The exchange matrix and the projective and injective dimensions are
\[ B=\begin{pmatrix}0&2&-2\\-2&0&2\\2&-2&0\end{pmatrix},\qquad P_i\simeq I_i,\qquad \dv P_i=(4,4,4),\qquad \dim_{\mathbb C}J=36. \]

We work in the \emph{generalized cluster category} $\mathcal C_{Q,W}$.
Write $T$ for its standard cluster-tilting object in the convention
$\CC_T(\Sigma T_i)=x_i$, and put $\Gamma=\Sigma T$.
The suspension transport used in this example is
Theorem~\ref{thm:triangulated-transport}; Section~\ref{sec:AR}
computes the resulting coordinate transformation.
We recall the obstruction to reaching $\Sigma\Gamma$ from $\Gamma$.
Under a mutation of the reference cluster-tilting object, the index
coordinates $g$ of a fixed rigid object satisfy, up to interchanging the
two neighbours of the mutated vertex $k$,
\[ g'_k=-g_k,\qquad g'_j=g_j+2[g_k]_+,\qquad g'_\ell=g_\ell-2[-g_k]_+. \]
Here $[a]_+=\max(a,0)$ is the \emph{positive part}; for vectors we
apply it coordinatewise. Every mutation of the Markov quiver reverses
its orientation, so the formula applies at each step and preserves
$g_1+g_2+g_3$.
The index of $\Sigma\Gamma$ relative to $\Gamma$ has sum $-3$;
relative to itself it has sum $3$. Hence the two cluster-tilting objects
cannot be joined by mutations \cite[Example~4.3]{Pla13}.
In Section~\ref{subsec:markov-characters} we compute the character
transformation between them, and in Section~\ref{sec:tropical} we compute
its tropicalization.

\section{Auslander--Reiten translation and \texorpdfstring{$F$}{F}-polynomials}\label{sec:AR}

We regard dimension vectors and weights as row vectors and set
\[ b_{ij}=\#\{i\to j\}-\#\{j\to i\},\qquad b_i=(b_{i1},\ldots,b_{im}). \]
Let $J$ be the finite-dimensional \emph{completed Jacobian algebra}
of a quiver with $Q_0=\{1,\ldots,m\}$ and no loops or oriented
two-cycles. The potential may be a formal series. We use right
$J$-modules and multiply paths in traversal order. We write
\[ P_i=e_iJ,\qquad I_i=\Hom_{\mathbb C}(Je_i,\mathbb C) \]
for the indecomposable projectives and injectives, where $e_i$ is the
vertex idempotent. We denote the corresponding simple module by $S_i$.
We reserve $\mathbf e_i$ for the standard basis
vector of $\ZZ^{Q_0}$ and $\mathrm{Id}_m$ for the identity matrix. For $\alpha\in\ZZ_{\geq0}^{Q_0}$, put
\[ P(\alpha)=\bigoplus_i P_i^{\oplus\alpha_i},\qquad I(\alpha)=\bigoplus_i I_i^{\oplus\alpha_i}. \]
For a weight $\delta=(\delta_i)$, we write
$\delta(\gamma)=\sum_i\delta_i\gamma_i$ and $\delta(i)=\delta_i$.

A \emph{decorated representation} is a pair $M=(M^{\mathrm{ord}},V_M)$,
where $M^{\mathrm{ord}}$ is a finite-dimensional $J$-module and
$V_M$ is a $Q_0$-graded vector space. We set
$\dv M=\dv M^{\mathrm{ord}}$ and $v_M=\dv V_M$.
Let
\[ P(\beta_1)\longrightarrow P(\beta_0)\longrightarrow M^{\mathrm{ord}}\longrightarrow0,
\qquad 0\longrightarrow M^{\mathrm{ord}}\longrightarrow I(\check\beta_0)\longrightarrow I(\check\beta_1) \]
be minimal projective and injective presentations. The \emph{weight} and
\emph{coweight} are
\[ \delta_M=\beta_0-\beta_1-v_M,\qquad
\check\delta_M=\check\beta_0-\check\beta_1-v_M. \]
Thus negative decoration enters both with the same sign.

We take $y_1,\ldots,y_m$ to be algebraically independent indeterminates.
The \emph{$F$-polynomial} and the \emph{dual $F$-polynomial} count
submodules and quotients, respectively \cite[Section~3]{DWZ10}:
\[ F_M(y)=\sum_{\gamma}\chi_c(\Gr_{\gamma}M^{\mathrm{ord}})y^{\gamma},\qquad \check F_M(y)=\sum_{\gamma}\chi_c(\Gr^{\gamma}M^{\mathrm{ord}})y^{\gamma}. \]
Here $\Gr_{\gamma}M^{\mathrm{ord}}$ parametrizes submodules of dimension
vector $\gamma$, and $\Gr^{\gamma}M^{\mathrm{ord}}$ parametrizes
quotients of that dimension.
Both polynomials depend only on the ordinary part of $M$, and the
submodule--quotient correspondence gives
\[ \check F_M(y^{-1})=y^{-\dv M}F_M(y),\qquad y^{-1}=(y_1^{-1},\ldots,y_m^{-1}). \]
For the \emph{extended Auslander--Reiten translation}, write
$M^{\mathrm{ord}}=M^{\mathrm{np}}\oplus P(\pi_M)$ with
$M^{\mathrm{np}}$ having no projective summand, and set
\[ \tau M=\bigl(\tau M^{\mathrm{np}}\oplus I(v_M),\pi_M\bigr). \]
Here $\tau M^{\mathrm{np}}$ is the kernel of the Nakayama transform of a
minimal projective presentation. Writing $(0,\mathbf e_i)$ for the
negative simple, we have $\tau P_i=(0,\mathbf e_i)$ and
$\tau(0,\mathbf e_i)=I_i$.

\Needspace{10\baselineskip}
\begin{theorem}[Auslander--Reiten translation]
\label{thm:AR}
Put
\[ F_I(z)^c=\prod_jF_{I_j}(z)^{c_j},\qquad \psi_i(z)=z_i^{-1}\prod_jF_{I_j}(z)^{b_{ij}}. \]
Let $z_1,\ldots,z_m$ be algebraically independent variables.
For every decorated representation $M$,
\begin{equation}\label{eq:AR} F_M(\psi(z))=z^{-\dv M}F_I(z)^{\check\delta_M}F_{\tau M}(z)\end{equation}
as an equality in $\QQ(z_1,\ldots,z_m)$.
\end{theorem}

The Markov quiver already shows why the ambient field in
\eqref{eq:AR} matters. Its
monomial substitution is
\begin{equation}\label{eq:markov-y} y_1=x_2^2/x_3^2,\qquad y_2=x_3^2/x_1^2,\qquad y_3=x_1^2/x_2^2,\end{equation}
so $y_1y_2y_3=1$. Thus an identity in cluster coordinates alone need
not prove \eqref{eq:AR}. We recover the full identity by adjoining sources
and setting only the newly added $F$-variables to zero. Below we will
see that the Markov transformation of the $y$-variables is nontrivial
before this specialization, but becomes the identity on
$y_1y_2y_3=1$.

\subsection{Cluster characters and suspension}
We apply Theorem~\ref{thm:triangulated-transport} in the generalized
cluster category $\C$ of the complete Ginzburg algebra. It is Hom-finite
and 2-Calabi--Yau, with cluster-tilting object $T$ satisfying
$\End_\C(T)\simeq J$ \cite[Theorem~3.5]{Ami09}; see also
\cite[Section~7.20]{KY11} for complete path algebras.

We use Palu's full multiplication theorem and constructibility of cones
\cite[Theorem~1.1 and Section~2.5]{Pal12}. In the completed setting, the
relevant Hom diagrams are represented by fixed finite-dimensional dg
truncations \cite[Theorem~2.19]{KY11}. For a fixed pair, one common
finite-dimensional replacement and a linear section represent the whole
morphism space; one power of the Ginzburg arrow ideal then annihilates
all representatives and their cones. Thus the constructibility argument
takes place over one finite-dimensional quotient for the entire family,
as in \cite[Section~2.5]{Pal12}. The same applies after adding sources.

Put $H=\C(T,-)$. Let $[\Sigma T]$ denote the ideal of morphisms
factoring through $\add\Sigma T$. The equivalence
$\C/[\Sigma T]\simeq\op{mod}J$
\cite[Proposition~2.1(c)]{KR07} gives the ordinary module part
$M^{\mathrm{ord}}=H(L)$; summands $\Sigma T_i$ give the negative
decoration. If $L$ corresponds to $M$, Palu's character is
\[ \CC_T(L;x)=x^{-c}F_M(y),\qquad c=\check\delta_M,\qquad y_i=x^{b_i}. \]
Here $\CC_T(\Sigma T_i;x)=x_i$, and a submodule of dimension $a$
contributes exponent $aB$ \cite{Pal08,Pal12}. Thus $T$ is the reference
object while $\Sigma T_i$ are the coordinate objects. We write $\ind_T$
and $\coind_T$ for the \emph{index} and \emph{coindex} relative to $T$.

Let $d=\dv M$ and $c'=\check\delta_{\tau M}$. The presentation
identities used below are
\begin{equation}\label{eq:app-weights} c=\delta_M+dB,\qquad c'=-\delta_M,\qquad c+c'=dB.\end{equation}
We prove the first identity using the index and coindex formulas of
\cite[Lemma~2.3]{Pal08}. Write
$\langle U,V\rangle=\dim\Hom_J(U,V)-\dim\Ext^1_J(U,V)$, and
let $\langle-,-\rangle_a$ be its antisymmetrization. This form descends
to $K_0(\op{mod}J)$ by \cite[Theorem~3.4]{Pal08}.
The decoration contributes equally to the index and coindex, so
\[ c_i-\delta_M(i) =\langle S_i,H(L)\rangle_a =\sum_j d_j\langle S_i,S_j\rangle_a =\sum_j d_j b_{ji}. \]
Thus the first identity holds for the completed Jacobian algebra, without
an assumption on the uncompleted Jacobian ideal. The second follows by
applying the \emph{Nakayama functor} $\nu$ to the decorated projective
presentation.
The same presentation, lifted to a triangle
\[ T_-\longrightarrow T_+\longrightarrow L\longrightarrow\Sigma T_-, \]
shows that $\Sigma L$ corresponds to $\tau M$. Indeed,
$H(\Sigma L)$ is the kernel of the induced map
$H(\Sigma^2T_-)\to H(\Sigma^2T_+)$, which is the Nakayama
transform of the presentation. This includes its decorated summands.

Taking $L=\Sigma T_i$, 2-Calabi--Yau duality gives
$H(\Sigma^2T_i)\simeq I_i$ and $\coind_T(\Sigma^2T_i)=\mathbf e_i$,
and therefore, with $y_j=x^{b_j}$,
\begin{equation}\label{eq:shifted-injective-character}
\CC_T(\Sigma^2T_i;x)=x_i^{-1}F_{I_i}(y).
\end{equation}
This formula does not require a rank assumption on $B$.

\subsection{The nonsingular case}
Assume first that $B$ is nonsingular. Introduce algebraically independent indeterminates
$\xi_1,\ldots,\xi_m$ and set $z_i=\xi^{b_i}$. Evaluating
\eqref{eq:shifted-injective-character} at $x_i=\xi_i$ gives
$\CC_T(\Sigma^2T_i;\xi)=\xi_i^{-1}F_{I_i}(z)$.
The assignments
\[ \phi:\QQ[x_1^{\pm1},\ldots,x_m^{\pm1}]\longrightarrow\QQ(\xi), \qquad x_i\longmapsto \xi_i^{-1}F_{I_i}(z) \]
define a homomorphism. Indeed, $a\mapsto aB$ is injective, so
$F_{I_i}(\xi^B)\ne0$.

We apply Theorem~\ref{thm:triangulated-transport} with
$\mathbb K=\QQ$ and $\mathbb F=\QQ(\xi)$. By the formula for
$\CC_T(\Sigma^2T_i;\xi)$, the homomorphism $\phi$ satisfies its
coordinate hypothesis. Therefore
\[ \phi(\CC_T(L;x))=\CC_T(\Sigma L;\xi) \qquad(L\in\C). \]
Since $\phi(y_i)=\psi_i(z)$, we obtain
\[ \xi^cF_I(z)^{-c}F_M(\psi(z))=\xi^{-c'}F_{\tau M}(z). \]
By \eqref{eq:app-weights}, the remaining monomial is
$\xi^{-(c+c')}=\xi^{-dB}=z^{-d}$. We have therefore proved
\eqref{eq:AR} after the substitution $z_i=\xi^{b_i}$.
Nonsingularity makes this substitution injective on rational-function
fields, so the equality holds in $\QQ(z_1,\ldots,z_m)$.

\subsection{Adding sources}
For each old vertex $i$, add a new source $i^*$ and one arrow
$i^*\to i$, keeping the potential unchanged. The enlarged exchange
matrix is
\[ \widetilde B= \begin{pmatrix}B&-\mathrm{Id}_m\\ \mathrm{Id}_m&0\end{pmatrix},\qquad \widetilde B^{-1}= \begin{pmatrix}0&\mathrm{Id}_m\\-\mathrm{Id}_m&B\end{pmatrix}. \]
The enlarged completed Jacobian algebra $\widetilde J$ is
finite-dimensional: a path involving a new vertex consists of one
new arrow followed by an old path. We may therefore apply the
nonsingular case to $\widetilde J$.

Let $\widetilde M$ be the extension of $M$, including its
decoration, by zero at the new sources. The old projectives are
unchanged as zero-extended modules, because no path goes from an old
vertex to a new one. Thus the same decorated minimal projective
presentation presents $\widetilde M$. Formula
\eqref{eq:app-weights} gives
\begin{equation}\label{eq:app-enlarged-weights} \dv\widetilde M=(d,0),\qquad \delta_{\widetilde M}=(\delta_M,0),\qquad \check\delta_{\widetilde M}=(c,-d).\end{equation}

Write $e$ for the sum of the old vertex idempotents and
$\res N=Ne$. For each old vertex,
$\res\widetilde I_i\simeq I_i$.
These identifications commute with maps between the injectives obtained
by applying the Nakayama functor to old projectives. Equivalently,
$\res\circ\nu_{\widetilde J}\circ\iota\simeq\nu_J$ on old
projectives, where $\iota$ is extension by zero. We apply this
natural isomorphism to the common decorated presentation and use
exactness of restriction. We obtain
\[ \res\bigl((\widetilde\tau\widetilde M)^{\mathrm{ord}}\bigr)\simeq(\tau M)^{\mathrm{ord}}. \]
This argument includes presentation summands $P_i\to0$ from the decoration.

Use variables $z$ at the old vertices and $w$ at the new sources.
For any $\widetilde J$-module $N$,
\[ \Gr_{(a,0)}N\simeq\Gr_a(Ne),\qquad F_N(z,0)=F_{Ne}(z). \]
Indeed, zero subspaces at the new sources impose no extra conditions
on submodules at the old vertices. It follows that
\[ F_{\widetilde I_i}(z,0)=F_{I_i}(z),\qquad F_{\widetilde\tau\widetilde M}(z,0)=F_{\tau M}(z). \]
At a new source the injective is simple, with $F$-polynomial
$1+w_i$. Put $\widetilde F_i(z,w)=F_{\widetilde I_i}(z,w)$ for old
vertices. The old coordinates of the enlarged substitution are
\[ \widetilde\psi_i(z,w) =z_i^{-1}\prod_j\widetilde F_j(z,w)^{b_{ij}}(1+w_i)^{-1}. \]
Since $F_{\widetilde M}$ involves only old variables, the
nonsingular case and \eqref{eq:app-enlarged-weights} give
\[ F_M(\widetilde\psi(z,w)) =z^{-d}\prod_i\widetilde F_i(z,w)^{c_i} \prod_i(1+w_i)^{-d_i}F_{\widetilde\tau\widetilde M}(z,w). \]
Both sides belong to
\[ \QQ(z)[w_1,\ldots,w_m] [\widetilde F_1^{-1},\ldots,\widetilde F_m^{-1},(1+w_1)^{-1},\ldots,(1+w_m)^{-1}]. \]
Every inverted element has nonzero image at $w=0$. The new
coordinate substitutions, which contain $w_i^{-1}$, do not occur
in the left-hand side; the dimension monomial on the right contains
no $w$-factor. We may therefore set $w=0$. This gives precisely
\eqref{eq:AR} in $\QQ(z_1,\ldots,z_m)$, and completes the proof
of Theorem~\ref{thm:AR}.

\subsection{The inverse and the character transformation}
\begin{proposition}\label{prop:AR-inverse}
The map $\psi$ is birational. Its inverse is given by the dual
$F$-polynomials of the indecomposable projectives:
\begin{equation}\label{eq:AR-inverse} \rho_i(y)=y_i^{-1}\prod_j\check F_{P_j}(y^{-1})^{b_{ij}}.\end{equation}
\end{proposition}

\begin{proof}
Put $p_j=\dv P_j$. The projective $P_j$ has weight $\mathbf e_j$,
coweight $\mathbf e_j+p_jB$, and $F_{\tau P_j}=1$. Applying
\eqref{eq:AR}, we obtain
\[ F_{P_j}(\psi(z))=z^{-p_j}F_I(z)^{\mathbf e_j+p_jB}=\psi(z)^{p_j}F_{I_j}(z). \]
The submodule--quotient relation therefore gives
$\check F_{P_j}(\psi(z)^{-1})=F_{I_j}(z)$. 
Substituting in \eqref{eq:AR-inverse}, we get
$\rho_i(\psi(z))=z_i$. Hence $\QQ(\psi_1,\ldots,\psi_m)$
contains every $z_i$ and equals $\QQ(z_1,\ldots,z_m)$. This proves
birationality and the inverse formula, without a rank condition on $B$.
\end{proof}

\begin{corollary}\label{cor:triangulated-character}
In the notation above, put $y_i=x^{b_i}$. Suppose that
$F_{I_i}(y)\ne0$ for every $i$. Then the assignments
$\Theta_T(x_i)=x_i^{-1}F_{I_i}(y)$
define an automorphism of $\mathbb C(x_1,\ldots,x_m)$, whose inverse is
$\Theta_T^{-1}(x_i)=x_i^{-1}\check F_{P_i}(y^{-1})$.
For every $L\in\C$, we have
\begin{equation}\label{eq:triangulated-character} \Theta_T\bigl(\CC_T(L;x)\bigr)=\CC_T(\Sigma L;x),\qquad \Theta_T^{-1}\bigl(\CC_T(L;x)\bigr)=\CC_T(\Sigma^{-1}L;x).\end{equation}
Both maps preserve the algebra
$\mathscr A_{\CC_T}=\mathbb C[\CC_T(L;x):L\in\C]$.
\end{corollary}

\begin{proof}
Write $q_i=x_i^{-1}F_{I_i}(y)$. Nonvanishing defines a homomorphism
$\phi:\mathbb C[x_1^{\pm1},\ldots,x_m^{\pm1}]\to\mathbb C(x)$.
By \eqref{eq:shifted-injective-character},
$q_i=\CC_T(\Sigma^2T_i;x)$. Theorem~\ref{thm:triangulated-transport} gives
$\phi\bigl(\CC_T(L;x)\bigr)=\CC_T(\Sigma L;x)$ for $L\in\C$.
This is the first identity in \eqref{eq:triangulated-character}.

Set $p_i=\dv P_i$. Since $P_i$ has coweight $\mathbf e_i+p_iB$,
\[ r_i:=\CC_T(T_i;x)=x_i^{-1}y^{-p_i}F_{P_i}(y)=x_i^{-1}\check F_{P_i}(y^{-1}). \]
Taking $L=T_i$ gives $\phi(r_i)=x_i$. Hence
$\mathbb C(q_1,\ldots,q_m)=\mathbb C(x_1,\ldots,x_m)$, and the
$q_i$ are algebraically independent. Thus $\phi$ extends to the
stated automorphism, with inverse $x_i\mapsto r_i$. Applying the first
identity to $\Sigma^{-1}L$ gives the second. Since suspension permutes
all objects, both automorphisms preserve $\mathscr A_{\CC_T}$.
\end{proof}

\subsection{The Markov quiver: the character involution}\label{subsec:markov-characters}
We return to \eqref{eq:markov-potential}, with indices read cyclically.
Put $q=z_1z_2z_3$ and
\begin{equation}\label{eq:markov-A} A_i(z)=1+z_i(1+z_{i-1})^2+z_iz_{i-1}^2z_{i+1}(2+z_{i+1})+q^2.\end{equation}
These are the generating polynomials for \emph{order ideals} in the
labelled posets of \cite[Example~1.3 and Figure~2]{Wen23}. Using the
projective modules of \cite[Example~4.3]{Pla13}, in our path convention,
we obtain
\begin{equation}\label{eq:markov-F} F_{P_i}(z)=F_{I_i}(z)=(1+q)^2A_i(z).\end{equation}
We hold the common factor because the theorem concerns the full
$F$-polynomials, not just their ratios.

To compute the common factor, we use the two alternating branches of
$\op{rad}P_1$, which have a common socle. A generic arrow grading satisfying
$\deg a_j+\deg b_j+\deg c_j=0$ separates its basis vectors at each
vertex. The torus-fixed proper nonzero submodules consist of the socle
and a suffix of each branch. Each suffix has generating polynomial
$(1+q)(1+z_3+z_2z_3)$; the two extra arrow maps exclude the suffix
pairs $(5,0)$ and $(0,5)$. Since Euler characteristic agrees with
that of the fixed locus, adding zero and the whole projective gives
\[ F_{P_1}=1+z_1(1+q)^2(1+z_3+z_2z_3)^2-2q^2+q^4=(1+q)^2A_1. \]
Cyclic symmetry gives the other two formulas. The same description of the projectives
shows that the linear form taking value $1$ on each socle path and
$0$ on the other path-basis elements is a \emph{symmetrizing form}: it pairs
complementary paths. We will use this identification of projectives and
injectives in the nonrigid example below.

The common factor in \eqref{eq:markov-F} cancels from the coordinate ratios:
\begin{equation}\label{eq:markov-psi} \psi_i(z)=z_i^{-1}\left(\frac{A_{i+1}(z)}{A_{i-1}(z)}\right)^2,\qquad q(\psi(z))=q^{-1}.\end{equation}
Since $P_i\simeq I_i$, their dimension vectors are all $(4,4,4)$,
and $\sum_jb_{ij}=0$ for every $i$, Proposition~\ref{prop:AR-inverse}
gives $\rho=\psi$. Thus $\psi$ is a subtraction-free involution.

Under \eqref{eq:markov-y}, we have $q=1$. Set
\[ \kappa=\frac{x_1^2+x_2^2+x_3^2}{x_1x_2x_3}. \]
Then \eqref{eq:markov-F} gives $F_{I_i}(y)=4\kappa^2x_i^2$, so the
character transformation is
\begin{equation}\label{eq:markov-theta} \Theta(x_i)=4\kappa^2x_i,\qquad \Theta(\kappa)=\frac1{4\kappa},\qquad \Theta^2=\mathrm{id}.\end{equation}
Corollary~\ref{cor:triangulated-character} now gives
\begin{equation}\label{eq:markov-covariance} \Theta\bigl(\CC_T(L;x)\bigr)=\CC_T(\Sigma L;x)\qquad(L\in\mathcal C_{Q,W}).\end{equation}
In particular, $\Theta$ preserves the algebra generated by all characters,
and $\CC_T(\Sigma^2L;x)=\CC_T(L;x)$. This is periodicity of character
values; it does not assert periodicity of the suspension functor.

We can express the involution in simpler rational coordinates. Set
\begin{equation}\label{eq:markov-ratio-coordinates} u=x_1/x_3,\qquad v=x_2/x_3,\qquad h=2\kappa.\end{equation}
They generate the same rational-function field, since
\[ x_3=\frac{2(u^2+v^2+1)}{uvh},\qquad x_1=ux_3,\qquad x_2=vx_3. \]
Formula \eqref{eq:markov-theta} becomes
\begin{equation}\label{eq:markov-inversion} \Theta(u)=u,\qquad \Theta(v)=v,\qquad \Theta(h)=h^{-1}.\end{equation}
Thus $\Theta$ is a single coordinate inversion in this rational chart.
We use these coordinates only in the function field; they do not
identify the character algebra with a Laurent polynomial ring.

The grading $\deg x_i=1$ makes every reachable cluster variable
homogeneous of degree $1$, since its exchange relation is
$x_ix_i'=x_j^2+x_k^2$. But $\Theta(x_i)$ has degree $-1$.
Thus the shifted initial characters are not reachable, and $\Theta$
does not preserve the ordinary cluster algebra.

For $\{i,j,k\}=\{1,2,3\}$, let $\mu_i$ be the field involution
\[ \mu_i(x_i)=\frac{x_j^2+x_k^2}{x_i},\qquad \mu_i(x_j)=x_j,\qquad \mu_i(x_k)=x_k. \]
Direct substitution gives $\mu_i(\kappa)=\kappa$. Since every coordinate
function of $\mu_i$ is homogeneous of degree one, we also have
\[ \Theta\circ\mu_i=\mu_i\circ\Theta. \]
Every reachable cluster is obtained by composing these involutions.
Thus $\kappa$ has the same Laurent expression in every reachable cluster,
and belongs to the upper cluster algebra $\mathcal U$. On the other hand,
\[ \Theta(\kappa)=\frac{x_1x_2x_3}{4(x_1^2+x_2^2+x_3^2)} \notin\kk[x_1^{\pm1},x_2^{\pm1},x_3^{\pm1}]. \]
Indeed, the numerator and denominator are coprime, and the denominator is
not a monomial. Hence $\Theta(\mathcal U)\not\subseteq\mathcal U$,
even though $\Theta$ commutes with every elementary mutation.

This failure is part of the known once-punctured-torus phenomenon. Zhou
constructs two non-equivalent cluster structures from distinct subfans of
the scattering diagram \cite[Section~5.3]{Zho20}. Our calculation supplies
the cluster-character interpretation: the involution is realized by
suspension on every character through Corollary~\ref{cor:triangulated-character}.

The two transformations are compatible with the monomial coordinates:
\[ \Theta(y_i)=\prod_j\Theta(x_j)^{b_{ij}}=y_i^{-1}\prod_jF_{I_j}(y)^{b_{ij}}=\psi_i(y). \]
Here we use $\Theta(x_j)=x_j^{-1}F_{I_j}(y)$. Since $\Theta$
fixes each ratio $x_i/x_j$, it fixes every $y_i$. The monomial
map \eqref{eq:markov-y} is dominant onto $q=1$; hence $\psi$
is the identity on this image, although it sends $q$ to $q^{-1}$
before specialization.

\subsection{Stability scattering and a theta-function comparison criterion}\label{subsec:chamber-products}
We compare the \emph{stability scattering diagram} of $J$, with
classical Euler-characteristic integration \cite{Bri17,CMQ24}, and the
cluster scattering diagram of \cite{GHKK18}. The chamber formulas below
reduce their equivalence to a finite family of theta-function identities.
The projective formula comes from framed Hall theory; we obtain its
injective dual directly from the AR identity.

In this subsection we introduce \emph{principal coefficients} $t_1,\ldots,t_m$ and work in
\[ \widehat R_{\mathrm{prin}}=\kk[x_1^{\pm1},\ldots,x_m^{\pm1}][[t_1,\ldots,t_m]],\qquad y_i=t_i x^{b_i}. \]
Thus the $y_i$ remain algebraically independent even when $B$ is
singular. The exponent map for a dimension vector $d$ is
$p^*(d)=(dB,d)$, so $x^{dB}t^d=y^d$. We continue to use right
$J$-modules; when applying \cite{CMQ24}, we identify them with left
$J^{\mathrm{op}}$-modules. Our exponent vectors are rows, whereas
\cite{CMQ24} uses columns; in that convention the mutable exchange
matrix corresponding to ours is $B^{\mathsf t}$.
Let
\[ C^+=\{\delta\in\RR^{Q_0}:\delta(\mathbf e_i)>0\text{ for all }i\},\qquad C^-=-C^+ \]
be the positive and negative chambers in the space of stability weights.
Write $\mathfrak p_{+\to-}$ and $\mathfrak p_{-\to+}$ for the
\emph{path-ordered products} in the stability diagram between these
chambers. We fix the sign by requiring a positive-to-negative crossing
of the simple wall $\mathbf e_i^\perp$ to act as
\[ x_j\longmapsto x_j(1+y_i)^{\delta_{ij}},\qquad t_j\longmapsto t_j. \]
Products are taken degree by degree in the $t$-adic completion;
no finite wall-crossing or mutation path is assumed.

\begin{proposition}\label{prop:chamber-products}
For the finite-dimensional completed Jacobian algebra $J$ of
Theorem~\ref{thm:AR}, the stability chamber products satisfy
\begin{equation}\label{eq:chamber-products} \mathfrak p_{+\to-}(x_i)=x_i\check F_{P_i}(y),\qquad \mathfrak p_{-\to+}(x_i)=\frac{x_i}{F_{I_i}(y)}.\end{equation}
Both products fix every $t_j$. In particular, they are rational, and
\begin{equation}\label{eq:chamber-y} \mathfrak p_{+\to-}(y_i)=y_i\prod_j\check F_{P_j}(y)^{b_{ij}},\qquad \mathfrak p_{-\to+}(y_i)=y_i\prod_jF_{I_j}(y)^{-b_{ij}}.\end{equation}
\end{proposition}

\begin{proof}
The projective formula is \cite[Lemma~2.3]{CMQ24}, following
\cite[Theorem~1.4]{Bri17}. For the reverse product, define the continuous
substitution $V(x_i)=x_i/F_{I_i}(y)$, $V(t_i)=t_i$. Then
$V(y)=\psi(y)^{-1}$, with componentwise reciprocals. The calculation
in the proof of Proposition~\ref{prop:AR-inverse} gives
\[ V\bigl(x_i\check F_{P_i}(y)\bigr)=\frac{x_i\check F_{P_i}(\psi(y)^{-1})}{F_{I_i}(y)}=x_i. \]
Both substitutions fix every $t_i$, so
$V=\mathfrak p_{+\to-}^{-1}=\mathfrak p_{-\to+}$.
Formula \eqref{eq:chamber-y} follows from $y_i=t_ix^{b_i}$.
\end{proof}

In particular, \eqref{eq:chamber-y} gives
$\mathfrak p_{-\to+}(y_i^{-1})=\psi_i(y)$ and $\mathfrak p_{+\to-}(y_i)=\rho_i(y^{-1})$.
Thus the AR substitution includes a coordinate inversion in addition
to wall crossing. The path product itself is congruent to the identity
in the completion. It fixes the auxiliary coefficients $t_i$, unlike
the projective-injective character coordinates of Theorem~\ref{thm:main}.

Write $\vartheta^{\mathrm{st}}_{\alpha,C}$ and
$\vartheta^{\mathrm{cl}}_{\alpha,C}$ for the \emph{theta functions} of
the two diagrams, with initial exponent $(\alpha,0)$, where $\alpha\in\ZZ^{Q_0}$,
expanded in the chamber $C$. All identities below are identities of
full formal chart expansions.
In the principal-coefficient setting, the identities below follow for
injective-reachable seeds from Qin's weak-genteelness theorem
\cite[Theorem~1.2.4]{Qin24}. Here Proposition~\ref{prop:chamber-products}
turns the same identities into a criterion without a reachability
hypothesis.

\begin{corollary}\label{cor:chamber-theta}
The stability scattering diagram of $J$ is equivalent to the cluster
scattering diagram with the same principal data if and only if
\begin{equation}\label{eq:theta-comparison} \vartheta^{\mathrm{cl}}_{-\mathbf e_i,C^+}=x_i^{-1}F_{I_i}(y)\qquad(1\leq i\leq m).\end{equation}
For the stability diagram, we always have
\begin{equation}\label{eq:stability-theta} \vartheta^{\mathrm{st}}_{-a,C^+}=x^{-a}\prod_iF_{I_i}(y)^{a_i},\qquad \vartheta^{\mathrm{st}}_{a,C^-}=x^a\prod_i\check F_{P_i}(y)^{a_i}\end{equation}
for every $a\in\ZZ_{\geq0}^{Q_0}$. Under the equivalent conditions above,
\eqref{eq:stability-theta} and the chamber formulas
\eqref{eq:chamber-products} hold for the cluster diagram as well.
\end{corollary}

\begin{proof}
The chamber-monomial argument of \cite[Proposition~3.8]{GHKK18} gives
\[ \vartheta^{\bullet}_{a,C^+}=x^a,\qquad \vartheta^{\bullet}_{-a,C^-}=x^{-a}\qquad(a\in\ZZ_{\geq0}^{Q_0},\ \bullet\in\{\mathrm{st},\mathrm{cl}\}). \]
The proof uses only nonnegative wall degrees and skew-symmetry; reversing
its inequalities gives the negative-chamber statement without reachability.
By theta transport \cite[Theorem~3.5]{GHKK18},
$\vartheta^{\bullet}_{-\mathbf e_i,C^+}=\mathfrak p^{\bullet}_{-\to+}(x_i^{-1})$.
Proposition~\ref{prop:chamber-products} evaluates the stability expression
as $x_i^{-1}F_{I_i}(y)$. Equivalent diagrams therefore satisfy
\eqref{eq:theta-comparison}. Conversely, these identities make the two
chamber products agree on every $x_i^{-1}$. Both fix the $t_i$,
so their continuous actions agree. The principal-coefficient action is
faithful, and a consistent diagram is determined up to equivalence by
its total chamber product \cite[Theorem~2.2 and Remark~2.4]{CMQ24}.
This proves the converse.

Finally, transporting $x^{-a}$ from $C^-$ to $C^+$, and
$x^a$ in the reverse direction, gives \eqref{eq:stability-theta}
by Proposition~\ref{prop:chamber-products}. Equivalent diagrams have
the same chamber formulas.
\end{proof}

Using the positive-to-negative product instead, the same proof gives the
equivalent projective criterion
\[ \vartheta^{\mathrm{cl}}_{\mathbf e_i,C^-}=x_i\check F_{P_i}(y)\qquad(1\leq i\leq m). \]

The comparison also holds in some nonreachable cases, including
once-punctured closed surfaces of genus at least two
\cite[Theorem~1.4]{CMQ24}.

The Markov example shows that the criterion is not automatic. Its
stability and cluster scattering diagrams are not equivalent
\cite[Theorem~1.4]{CMQ24}, whereas \eqref{eq:markov-F} gives
\[ \vartheta^{\mathrm{st}}_{-\mathbf e_i,C^+}=x_i^{-1}(1+y_1y_2y_3)^2A_i(y). \]
By Corollary~\ref{cor:chamber-theta}, these cannot all be the corresponding
cluster theta functions. The common factor cancels from the $y_i$-images
in \eqref{eq:chamber-y}, since each row of $B$ sums to zero. Thus the
framed theta identities detect information that the induced
$y$-coordinate transformation can lose.

\section{Positroid, Grassmannian, and unipotent-cell twists}\label{sec:geometric}

We prove the twist formula for every object of the positroid Frobenius
category. The Grassmannian formula is its specialization to the uniform
positroid; we treat unipotent cells by the same argument.
In the twist formulas, we work with the indicated boundary functions
inverted. For Grassmannians and positroids, we use the coordinate ring
of the affine cone with the boundary Pl\"ucker coordinates inverted,
not just the field of degree-zero functions on the projective variety.

We use the \emph{source convention} of \cite{CKP24,Pre}: the
\emph{circle algebra} $C_{k,n}$ has relation $\mathsf y^k=\mathsf x^{n-k}$, and
$\mathsf x_i$ acts by $t$ on the \emph{rank-one object} $L_I$ when
$i\in I$, and by $1$ otherwise. For a positroid $\Pcal$, let
$\mathcal B_\Pcal$ be its \emph{boundary order} in this convention,
and put
\[ \E_\Pcal=\GP\mathcal B_\Pcal. \]
Here $\CM$ denotes \emph{Cohen--Macaulay modules}, finite free over
$\mathbb C[[t]]$, and $\GP$ denotes \emph{Gorenstein-projective
modules} over the indicated order.
We write $\Phi_\Pcal$ for its source Fu--Keller character, normalized
by $\Phi_\Pcal(L_I)=\Delta_I$ on a source Postnikov cluster.
Its projective-injective summands have the distinct
\emph{source-necklace} labels $\mathcal I^{\mathrm{src}}_\Pcal$.
We use one frozen variable for each such isomorphism class.
Appendix~\ref{app:positroid} compares the circle conventions and specifies
the boundary order also for disconnected necklaces.

We use the mutually inverse normalized Muller--Speyer twists
$\tw_L,\tw_R$ in the source convention \cite{MS}. In a Hom-finite
2-Calabi--Yau triangulated category, $\tau\simeq\Sigma$, so
\[ \overline{\Omega M}\simeq\tau^{-1}\bar M,\qquad \overline{\Omega^{-1}M}\simeq\tau\bar M. \]
With our convention, the left twist corresponds to inverse
Auslander--Reiten translation, and the right twist to forward translation.
The projective-injective factors hold the information lost on passing
to the stable category.

\begin{theorem}[Geometric twist formula]\label{thm:geometric}
The following formulas describe the syzygy transformation of character values.

\smallskip\noindent
\textup{(a) Positroids.}
For every positroid $\Pcal$ and every $M\in\E_\Pcal$,
\begin{equation}\label{eq:positroid} \tw_L^*\Phi_\Pcal(M)= \frac{\Phi_\Pcal(\Omega_\Pcal M)}{\Phi_\Pcal(P_M)},\qquad \tw_R^*\Phi_\Pcal(M)= \frac{\Phi_\Pcal(\Omega_\Pcal^{-1}M)}{\Phi_\Pcal(I_M)}.\end{equation}
Both the presentations and the syzygies are computed in
$\E_\Pcal$. In particular,
$\tw_L^*\Delta_J=\tw_R^*\Delta_J=\Delta_J^{-1}$ for
$J\in\mathcal I^{\mathrm{src}}_\Pcal$.

\smallskip\noindent
\textup{(b) Grassmannians.}
For the uniform positroid, $\E_\Pcal=\CM(C_{k,n})$ and
$\Phi_\Pcal$ is the Grassmannian character $\Psi$. Thus, for every
$M\in\CM(C_{k,n})$,
\begin{equation}\label{eq:grass} \tw_L^*\Psi(M)=\frac{\Psi(\Omega M)}{\Psi(P_M)},\qquad \tw_R^*\Psi(M)=\frac{\Psi(\Omega^{-1}M)}{\Psi(I_M)}.\end{equation}

\smallskip\noindent
\textup{(c) Unipotent cells.}
Let $Q$ be a finite acyclic quiver, $\Lambda$ its preprojective
algebra, and $w$ an element of the Weyl group of the associated
symmetric Kac--Moody datum. Let $\mathcal C_w$ be the
Geiss--Leclerc--Schr\"oer Frobenius category and $\varphi$ its
character on the unipotent cell $N^w$. For every $M\in\mathcal C_w$,
\begin{equation}\label{eq:unipotent} (\eta_w^*)^{-1}\varphi_M= \frac{\varphi_{\Omega_w M}}{\varphi_{P_M}},\qquad \eta_w^*\varphi_M= \frac{\varphi_{\Omega_w^{-1}M}}{\varphi_{I_M}}.\end{equation}
The projective-injective terms here are relative to $\mathcal C_w$,
not projective covers in the ambient category of $\Lambda$-modules.
\end{theorem}

Both positroid identities use the same source character and Frobenius
category. We do not extend them to $\CM\mathcal B_\Pcal$ outside
$\GP\mathcal B_\Pcal$, nor identify source and target characters there.

\begin{proof}[Proof of Theorem~\ref{thm:geometric}]
\emph{Multiplication formulas and character comparisons.}
For $\mathcal C_w$, the character $\varphi$ satisfies the full
preprojective multiplication formula \cite[Theorem~1]{GLS07}; extension
spaces and middle terms are those of this extension-closed subcategory.
Its agreement with the Fu--Keller character for every object is
\cite[Theorem~4]{GLS12}.

For $\CM(C_{k,n})$, we use the Grassmannian character obtained
by homogenizing the preprojective character \cite[Section~9]{JKS16}.
The full multiplication formula follows by the same homogenization:
all middle terms have rank equal to the sum of the ranks of the end
terms. The formula, including its constructible strata, is recorded in
\cite[equation~(4.18)]{JRS26}. For a reachable cluster-tilting object
$T$, its Laurent expression agrees with the Fu--Keller character
on every object by \cite[Theorem~9.11]{JKS24}. We choose such a
Postnikov cluster-tilting object; no reachability assumption is imposed
on the object whose character we evaluate.

For general boundary orders, we use the full exact embedding of
$\GP\mathcal B_\Pcal$ in the Grassmannian category and the cluster
subcategory theorem of Jensen--Riordan--Su
\cite[Theorems~2.1 and 3.11]{JRS26}, transported along our change of
conventions. In particular, its extension spaces and middle terms are
those of the ambient category. The character comparison is
\cite[Proposition~4.1 and Corollary~4.2]{JRS26}, and the localized
coordinate-ring identification is \cite[Theorem~7.9]{JRS26}.
We restrict the full formula \cite[equation~(4.18)]{JRS26} to this
extension-closed subcategory and then apply the character comparison.
This proves \eqref{eq:full} for $\Phi_\Pcal$, with constructible
integrands. These results include disconnected necklaces.

\smallskip\noindent\emph{Positroids and Grassmannians.}
For a connected Postnikov diagram, \cite[Theorem~12.2]{CKP24} gives
\[ \tw_L^*(\Delta_I)= \frac{\Phi_\Pcal(\Omega L_I)}{\Phi_\Pcal(P_{L_I})} \]
for rank-one modules. In particular this holds for every summand of one
source cluster-tilting object, including its projective summands.
Corollary~\ref{cor:seed} gives the first identity in
\eqref{eq:positroid} for every object. Its inverse gives the second.
The full Grassmannian is the uniform-positroid case, in which the
boundary order is $C_{k,n}$, so the same argument proves
\eqref{eq:grass}.

The initial identities for disconnected positroids are proved in
Appendix~\ref{app:positroid}. Corollary~\ref{cor:seed} then applies
without a connectedness assumption.

\smallskip\noindent\emph{Unipotent cells.}
We choose the cluster-tilting object $V_{\mathbf i}$ associated to a
reduced word for $w$. The calculation of twisted initial minors in
\cite[Section~9.3]{GLS12} identifies
\[ (\eta_w^*)^{-1}(\varphi_{V_{\mathbf i,j}}) =\frac{\varphi_{\Omega_wV_{\mathbf i,j}}} {\varphi_{P_{V_{\mathbf i,j}}}}. \]
Corollary~\ref{cor:seed} gives \eqref{eq:unipotent} for every object,
recovering \cite[Theorem~6]{GLS12}.
\end{proof}

Setting all frozen variables to $1$ in the character identities gives
covariance under suspension in the stable category. This assertion concerns
specialization of the character identities, not unrestricted specialization
of arbitrary rational functions.

\subsection*{The Marsh--Scott normalization}
For a rank-one Grassmannian module $L_I$, the particular nonminimal
projective presentation used in \cite[Sections~9 and 11]{CKP24} is
\[ 0\longrightarrow\Omega^\circ L_I\longrightarrow P_I^\circ \longrightarrow L_I\longrightarrow0. \]
For the original Marsh--Scott twist \cite{MSc}, we use the identity
of \cite[equation~(11.2)]{CKP24}
\[ \tw_{\mathrm{MS}}^*(\Delta_I)=\Psi(\Omega^\circ L_I). \]
Consequently, \eqref{eq:grass} gives
\[ \tw_{\mathrm{MS}}^*(\Delta_I) =\Psi(P_I^\circ)\,\tw_L^*(\Delta_I). \]
Thus the two normalizations differ by the frozen monomial
$\Psi(P_I^\circ)$. To obtain the perfect-matching expansion of the
rank-one character, we use \cite[Theorem~10.3]{CKP24}.

\subsection{Partition-function algebras}\label{subsec:partition-algebras}
Fix a positroid $\Pcal$ and a source Postnikov cluster-tilting object
$T\in\E_\Pcal$. Let $\mathfrak a_\Pcal$ be the homogeneous ideal of the
positroid variety, and let $S_\Pcal$ be the multiplicative set
generated by the boundary Pl\"ucker coordinates. Write $\overline S_\Pcal$
for its image in $\kk[\Gr(k,n)]/\mathfrak a_\Pcal$. We set
\[ R_\Pcal=\bigl(\kk[\Gr(k,n)]/\mathfrak a_\Pcal\bigr)[\overline S_\Pcal^{-1}] \]
and use the natural projection
\[ p:\kk[\Gr(k,n)][S_\Pcal^{-1}]\longrightarrow R_\Pcal. \]
The character comparison \cite[Proposition~4.1 and Corollary~4.2]{JRS26}
gives $\Phi_\Pcal(M)=p(\Psi(M))$ for $M\in\E_\Pcal$, where
$\Psi$ is the Grassmannian character in our source convention.

For $M\in\E_\Pcal$, we set
$\mathfrak P_M=\Psi(\Omega_\Pcal M)/\Psi(P_M)$, the
\emph{generalized partition function} of \cite[equation~(7.6)]{JRS26}.
Its denominator is a boundary monomial, so \eqref{eq:positroid} gives
\begin{equation}\label{eq:partition-image} \tw_L^*\Phi_\Pcal(M)=p(\mathfrak P_M)=\frac{\Phi_\Pcal(\Omega_\Pcal M)}{\Phi_\Pcal(P_M)}.\end{equation}
Thus the partition function, after projection, is the twist of the
character on all of $\E_\Pcal$. This answers
\cite[Remark~7.7]{JRS26} on that category; its cluster-character
property was proved in \cite[Proposition~7.6]{JRS26}.

We distinguish the algebras generated by all character values,
\[ \mathscr A_\Psi=\kk[\Psi(M):M\in\E_\Pcal],\qquad \mathscr A_{\mathfrak P}=\kk[\mathfrak P_M:M\in\E_\Pcal], \]
from the cluster algebras $\mathcal A_T$ and
$\mathcal A_T^{\mathfrak P}$ with initial coordinates $\Psi(T_i)$
and $\mathfrak P_{T_i}$, respectively, and the same exchange matrix.
These are the character and partition algebras of
\cite[Section~7.4]{JRS26}. We include the frozen generators, but do not
adjoin their inverses automatically. In particular, the partition seed
has frozen generators $\mathfrak P_{T_f}=\Psi(T_f)^{-1}$.

\begin{corollary}\label{cor:partition-algebra}
For every positroid, $\mathscr A_{\mathfrak P}=\mathcal A_T^{\mathfrak P}$.
The assignments on all character values define an isomorphism
\[ \widetilde{\tw}_L:\mathcal A_T\xrightarrow{\sim}\mathcal A_T^{\mathfrak P},\qquad \widetilde{\tw}_L(\Psi(M))=\mathfrak P_M\quad(M\in\E_\Pcal), \]
characterized by
\begin{equation}\label{eq:partition-intertwining} p\circ\widetilde{\tw}_L=\tw_L^*\circ p.\end{equation}
The equality and the isomorphism hold before further boundary localization,
without a reachability assumption on the shifted cluster-tilting object.
\end{corollary}

\begin{proof}
By \cite[Theorem~7.9]{JRS26}, $\mathscr A_\Psi=\mathcal A_T$, and
$p$ is injective on both $\mathscr A_\Psi$ and
$\mathscr A_{\mathfrak P}$. Set $\theta=\tw_L^*$. Formula
\eqref{eq:partition-image} gives
$\theta\bigl(p(\mathscr A_\Psi)\bigr)=p(\mathscr A_{\mathfrak P})$.
On the initial seed, $\theta$ sends the character coordinates to the
partition coordinates, including the frozen ones. Since a field
automorphism commutes with every exchange relation, induction on a
mutation word gives
$\theta\bigl(p(\mathcal A_T)\bigr)=p(\mathcal A_T^{\mathfrak P})$. 
The equality $\mathscr A_\Psi=\mathcal A_T$ identifies these two
images. The cluster-character property gives
$\mathcal A_T^{\mathfrak P}\subseteq\mathscr A_{\mathfrak P}$, so
injectivity of $p$ on the latter proves their equality.

We can therefore define
\[ \widetilde{\tw}_L=(p|_{\mathscr A_{\mathfrak P}})^{-1}\circ\theta\circ(p|_{\mathscr A_\Psi}), \]
where the inverse is taken on the image of the restricted map.
This is an algebra isomorphism, and \eqref{eq:partition-image} gives
its value on every $\Psi(M)$. Injectivity of $p$ also shows that
\eqref{eq:partition-intertwining} characterizes it uniquely.
\end{proof}

This identifies the seed isomorphism of \cite[Theorem~7.9]{JRS26}
with the geometric twist after projection and determines it on every
character. The partition-algebra equality removes the sufficient
reachability condition in \cite[Remark~7.13]{JRS26}; it does not prove
reachability itself. The two algebras need not be the same subalgebra
of the ambient function field, and their frozen generators are reciprocal.

\section{Tropical cluster characters and Auslander--Reiten translation}\label{sec:tropical}

We use the maximum convention of \cite[Sections~3.1 and 6.2]{FeiT23}.
We write $\omega\in\RR^n$ for a weight on the cluster-coordinate
exponent lattice, and $\delta\in\RR^{Q_0}$ for a weight on the
quiver dimension-vector lattice. In the pairing formulas,
$\gamma\in\ZZ^{Q_0}$ is an exponent vector, not an evaluation weight.

\subsection{Tropicalization and character identities}
For a nonzero Laurent polynomial $G=\sum_a c_ax^a$, put
\[ G^{\mathrm{trop}}(\omega)=\max_{c_a\ne0}\omega(a). \]
This is the \emph{support function} of the \emph{Newton polytope}
$\Newt(G)$, the convex hull of the exponents of $G$. For a nonzero rational
function $G=P/Q$, we set
$G^{\mathrm{trop}}=P^{\mathrm{trop}}-Q^{\mathrm{trop}}$.
The identity
$(PQ)^{\mathrm{trop}}=P^{\mathrm{trop}}+Q^{\mathrm{trop}}$
makes this well-defined. These functions are
continuous and piecewise linear; if $G_1+G_2\ne0$, then
\[ (G_1+G_2)^{\mathrm{trop}}\leq\max(G_1^{\mathrm{trop}},G_2^{\mathrm{trop}}), \]
with equality wherever the two terms have different degrees.

We write
$\CC^{\mathrm{trop}}_M:=(\CC(M))^{\mathrm{trop}}$
for the tropical cluster character. For a triangulated character, we
keep the reference object and write
$\CC^{\mathrm{trop}}_{T,L}:=\CC_T(L;x)^{\mathrm{trop}}$.
If $\CC(M)=0$, we set
$\CC^{\mathrm{trop}}_M=-\infty$; statements about Newton polytopes
concern nonzero characters.

A rational function is subtraction-free if it is a ratio of nonzero
Laurent polynomials with positive real coefficients. For such expressions,
tropicalization replaces multiplication, division, and addition by addition,
subtraction, and maximum, respectively, and commutes with substitution.
The following lemma allows arbitrary coefficients in the function being
transformed.

\begin{lemma}\label{lem:tropical-composition}
Let $\phi$ be a $\mathbb C$-automorphism of
$\mathbb C(x_1,\ldots,x_n)$. Put $q_i=\phi(x_i)$, $r_i=\phi^{-1}(x_i)$,
$q=(q_1,\ldots,q_n)$, $r=(r_1,\ldots,r_n)$, and
\[ D(\omega)=(q_1^{\mathrm{trop}}(\omega),\ldots,q_n^{\mathrm{trop}}(\omega)),\qquad \widehat D(\omega)=(r_1^{\mathrm{trop}}(\omega),\ldots,r_n^{\mathrm{trop}}(\omega)). \]
If $D$ is a homeomorphism of $\RR^n$, then
\begin{equation}\label{eq:tropical-composition} (\phi(G))^{\mathrm{trop}}(\omega)=G^{\mathrm{trop}}(D(\omega))\end{equation}
for every nonzero rational function $G$. Moreover, $D,\widehat D$ are inverse
integral piecewise-linear homeomorphisms. Their linear parts on
full-dimensional chambers are \emph{unimodular}.
In particular, these conclusions hold whenever all the coordinates
$q_i,r_i$ are subtraction-free.
\end{lemma}

\begin{proof}
Assume first that $D$ is a homeomorphism, and let
$G=\sum_a c_ax^a$ be a Laurent polynomial. Outside the finitely many
hyperplanes $\eta(a-b)=0$, with $a,b$ distinct exponents of $G$,
one monomial has strictly largest $\eta$-degree. The inverse image of
this set under $D$ is dense. At each point of that inverse image, one
term of $\sum_a c_aq^a$ has strictly largest degree and cannot cancel.
This proves \eqref{eq:tropical-composition} on a dense set, and continuity
proves it everywhere. We obtain the rational case by applying the result
to a numerator and a denominator.

Taking $G=r_i$ in \eqref{eq:tropical-composition} gives
$\widehat D\circ D=\mathrm{id}$. Since $D$ is a homeomorphism,
$\widehat D=D^{-1}$. Both maps are integral and piecewise linear. After a
common refinement of their chambers, their linear parts are inverse
integer matrices, hence unimodular.

Finally, suppose that all $q_i,r_i$ are subtraction-free. We
can tropicalize the identities $r(q(x))=x$ and $q(r(x))=x$ to
obtain $\widehat D\circ D=D\circ \widehat D=\mathrm{id}$. Thus $D$ is a
homeomorphism, and the first part applies.
\end{proof}

Subtraction-free coordinates in one direction alone do not suffice:
for $\phi(x)=x+y$, $\phi(y)=y$, we have
\[ D(u,v)=(\max(u,v),v),\qquad (\phi(x-y))^{\mathrm{trop}}(u,v)=u,\qquad (x-y)^{\mathrm{trop}}(D(u,v))=\max(u,v). \]
Here the inverse uses subtraction, and $D$ collapses the region $u<v$.

We now apply the lemma to the character transformations of
Theorem~\ref{thm:main}. Write
\[ \CC(P_M)=x^{p_M},\qquad \CC(I_M)=x^{i_M}, \]
where $p_M,i_M$ are supported on the frozen vertices.

\Needspace{12\baselineskip}
\begin{corollary}\label{cor:tropical-syzygy}
Define
\[ D_-(\omega)_i=\CC^{\mathrm{trop}}_{\Omega T_i}(\omega)-\omega(p_{T_i}),\qquad D_+(\omega)_i=\CC^{\mathrm{trop}}_{\Omega^{-1}T_i}(\omega)-\omega(i_{T_i}). \]
If either $D_-$ or $D_+$ is a homeomorphism, then they are inverse integral
piecewise-linear homeomorphisms, with $D_\pm(\omega)_f=-\omega_f$ at
each frozen vertex. For every object $M$,
\begin{equation}\label{eq:tropical-syzygy} \CC^{\mathrm{trop}}_{\Omega M}(\omega)=\CC^{\mathrm{trop}}_M(D_-(\omega))+\omega(p_M),\qquad \CC^{\mathrm{trop}}_{\Omega^{-1}M}(\omega)=\CC^{\mathrm{trop}}_M(D_+(\omega))+\omega(i_M).\end{equation}
\end{corollary}

\begin{proof}
The maps $D_-$ and $D_+$ are the tropicalizations of the inverse field
automorphisms $\mathsf t_-$ and $\mathsf t_+$. Apply
Lemma~\ref{lem:tropical-composition} to whichever of $D_-$ or $D_+$ is a
homeomorphism. The lemma identifies the other map with its inverse, so it
applies in both directions to \eqref{eq:objectwise} and gives
\eqref{eq:tropical-syzygy}. The frozen-coordinate assertion follows from
$\mathsf t_\pm(x_f)=x_f^{-1}$. If $\CC(M)=0$, the corresponding
character identity makes both translated characters zero, so the formulas
hold with our convention.
\end{proof}

\begin{remark}\label{rem:geometric-tropical}
Corollary~\ref{cor:tropical-syzygy} applies to all three geometric twists in
Theorem~\ref{thm:geometric}. For positroids, the twist and its inverse
are quasi-cluster automorphisms of the source structure
\cite[Theorem~7.2 and Proposition~7.3]{Pre}. Their initial coordinate
images are cluster variables times frozen Laurent monomials, so both
coordinate substitutions admit subtraction-free expressions obtained by
mutation. This includes the Grassmannian case. For unipotent cells, the
syzygy cluster is reachable \cite[Section~4.1]{GLS12}; applying inverse
suspension to a mutation path gives reachability of the cosyzygy cluster as
well. In each case Lemma~\ref{lem:tropical-composition} therefore verifies
the homeomorphism hypothesis of Corollary~\ref{cor:tropical-syzygy}. We
use reachability only for this coordinate calculation: the object $M$ is
arbitrary, and its character need not have positive coefficients.
\end{remark}

Formula \eqref{eq:tropical-syzygy} determines the Newton polytopes of
the translated characters through their support functions; it does not
assert that a single linear map carries one Newton polytope to the other.

\begin{example}[A tropical involution of the Markov quiver]
For $\omega=(\omega_1,\omega_2,\omega_3)$, write
\[ r(\omega)=\kappa^{\mathrm{trop}}(\omega)=2\max(\omega_1,\omega_2,\omega_3)-\omega_1-\omega_2-\omega_3. \]
Tropicalizing \eqref{eq:markov-theta}, we obtain
$D(\omega)=\omega+2r(\omega)\mathbf1$ where $\mathbf1=(1,1,1)$. 
Adding the same number to each coordinate preserves the largest-coordinate
chamber, and $r(D(\omega))=-r(\omega)$. Hence $D^2=\mathrm{id}$.
On the chamber where $\omega_k$ is largest, its matrix is
$\mathrm{Id}_3+2(2\mathbf e_k^{\mathsf t}-\mathbf1^{\mathsf t})\mathbf1$, an integral
reflection of determinant $-1$.

The rational coordinates \eqref{eq:markov-ratio-coordinates} give the
piecewise-linear coordinates
\[ a=\omega_1-\omega_3,\qquad b=\omega_2-\omega_3,\qquad r=r(\omega). \]
We recover the original weights from
\[ \omega_3=2\max(a,b,0)-a-b-r,\qquad \omega_1=a+\omega_3,\qquad \omega_2=b+\omega_3. \]
Thus the coordinate change and its inverse are integral and piecewise
linear. In these coordinates, \eqref{eq:markov-inversion} becomes simply
$(a,b,r)\mapsto(a,b,-r)$.
The two ratio weights are fixed and the weight of $h$ changes sign.

Since $D^2=\mathrm{id}$, Lemma~\ref{lem:tropical-composition} and
\eqref{eq:markov-covariance} give
$\CC^{\mathrm{trop}}_{T,\Sigma L}(\omega)=\CC^{\mathrm{trop}}_{T,L}(D(\omega))$
for every object of the Markov cluster category. This is the analogue of Corollary~\ref{cor:tropical-syzygy} in a
case where the shifted cluster is not reachable.
\end{example}
\subsection{Tropical \texorpdfstring{$F$}{F}-polynomials and Auslander--Reiten translation}\label{subsec:tropical-AR}
For the Jacobian algebra of Theorem~\ref{thm:AR}, we use the notation of
\cite[Definitions~3.1--3.2]{FeiT23}:
\[ f_M(\delta)=\max_{L\hookrightarrow M}\delta(\dv L),\qquad \check f_M(\delta)=\max_{M\twoheadrightarrow N}\delta(\dv N). \]
Here $\delta\in\RR^{Q_0}$; submodules and quotients refer to
$M^{\mathrm{ord}}$. Their Newton polytopes satisfy
\[ \NP(M)=\op{conv}\{\dv L:L\hookrightarrow M\},\qquad \check\NP(M)=\dv M-\NP(M). \]
By \cite[Theorem~4.2]{FeiC23}, $\NP(M)=\Newt(F_M)$. Thus
$f_M=f_{F_M}$ and $\check f_M=f_{\check F_M}$, without positivity
of their coefficients. We have
\[ f_M(\delta)-\check f_M(-\delta)=\delta(\dv M),\qquad \mathbf f_I(\delta)=(f_{I_1}(\delta),\ldots,f_{I_m}(\delta)). \]

We write $\hom(L,M)$ for Hom dimension between ordinary
parts. The following formulas describe the change of $f_M$ and
$\hom(L,M)$ under translation.

\begin{proposition}\label{prop:AR-transformations}\label{cor:tropical-AR}\label{prop:hom-AR}
\textup{(a)} Define
\[ \mathcal D(\delta)(i)=-\delta(i)+\sum_jb_{ij}f_{I_j}(\delta),\qquad
\widehat{\mathcal D}(\delta)(i)=-\delta(i)+\sum_jb_{ij}\check f_{P_j}(-\delta). \]
These are the tropicalizations of $\psi$ and $\rho$, respectively.
If $\mathcal D$ is a homeomorphism of $\RR^{Q_0}$, then
$\widehat{\mathcal D}=\mathcal D^{-1}$, and both are integral
piecewise-linear homeomorphisms. For every decorated representation
$M$ and every $\delta\in\RR^{Q_0}$,
\begin{equation}\label{eq:tropical-AR} f_{\tau M}(\delta)-f_M(\mathcal D(\delta))=\delta(\dv M)-\check\delta_M\bigl(\mathbf f_I(\delta)\bigr).\end{equation}
\textup{(b)} For any decorated representations $L,M$, with no
hypothesis on $\mathcal D$,
\begin{equation}\label{eq:hom-AR} \hom(\tau L,\tau M)-\hom(L,M)=\delta_{\tau L}(\dv M)-\check\delta_M(\dv\tau L).\end{equation}
\end{proposition}

\begin{proof}
For (a), the identities $f_M=f_{F_M}$ and
$\check f_M=f_{\check F_M}$ above give
\[ F_{I_j}^{\mathrm{trop}}=f_{I_j},\qquad
\check F_{P_j}^{\mathrm{trop}}=\check f_{P_j}. \]
Hence the displayed maps are the tropicalizations of $\psi$ and $\rho$.
By Proposition~\ref{prop:AR-inverse}, $\rho=\psi^{-1}$. If $\mathcal D$
is a homeomorphism, Lemma~\ref{lem:tropical-composition}
gives $\widehat{\mathcal D}=\mathcal D^{-1}$ and allows us to tropicalize
\eqref{eq:AR}, which gives \eqref{eq:tropical-AR}.

For (b), the projective presentation of $L$ and the injective
presentation of $\tau M$ give, by Nakayama duality,
\[
\begin{aligned}
\hom(L,M)-\delta_L(\dv M)&=\hom(M,\tau L),\\
\hom(\tau L,\tau M)-\check\delta_{\tau M}(\dv\tau L)&=\hom(M,\tau L).
\end{aligned}
\]
These are the presentation forms of AR duality, including negative
decorations \cite[Lemma~2.3]{FeiT23}. We subtract the two equalities
and use $\check\delta_{\tau M}=-\delta_M$ to obtain
\[ \hom(\tau L,\tau M)-\hom(L,M)=-\delta_L(\dv M)-\delta_M(\dv\tau L). \]
Substituting $\delta_L+\delta_{\tau L}=-(\dv\tau L)B$ and
$\check\delta_M=\delta_M+(\dv M)B$ from \eqref{eq:app-weights},
we obtain \eqref{eq:hom-AR} by skew-symmetry of $B$.
\end{proof}

No positivity hypothesis on the coefficients of the $F$-polynomials is
needed in Proposition~\ref{prop:AR-transformations}(a). Their Newton
polytopes determine the tropicalizations above. Moreover, the
coefficients at the vertices are equal to $1$. The homeomorphism
hypothesis on $\mathcal D$ is used only to prevent cancellation under
composition, as in Lemma~\ref{lem:tropical-composition}. For instance,
subtraction-free coordinate expressions for both $\psi$ and $\rho$ are
a sufficient condition for this hypothesis, but are not part of the
statement.

\subsubsection*{The invariant pairings}

When the coweight of $M$ lies in the image of $B$, the injective
$F$-polynomial factor in \eqref{eq:AR} can be absorbed into a monomial
normalization of $F_M$. Let $\check\delta_M=\gamma B$, with $\gamma\in\ZZ^{Q_0}$, and put
$\gamma'=\dv M-\gamma$.
Multiplying \eqref{eq:AR} by
$\psi(z)^{-\gamma}=z^\gamma F_I(z)^{-\gamma B}$ cancels the
injective factors:
\begin{equation}\label{eq:normalized-F} \left(y^{-\gamma}F_M(y)\right)\big|_{y=\psi(z)}=z^{-\gamma'}F_{\tau M}(z).\end{equation}
Under the hypotheses of Proposition~\ref{prop:AR-transformations}(a),
Lemma~\ref{lem:tropical-composition} applied to \eqref{eq:normalized-F}
gives
\begin{equation}\label{eq:invariant-pairing} f_M(\mathcal D(\delta))-\mathcal D(\delta)(\gamma)=f_{\tau M}(\delta)-\delta(\gamma').\end{equation}
We now pass from individual representations to their principal
components.

For an integral weight $\eta$, a decorated representation is \emph{general of
weight} $\eta$ if its decorated presentation is general in
\[ \op{PHom}(\eta):=\Hom_J(P([-\eta]_+),P([\eta]_+)). \]
We write $\dv(\eta)$ for its ordinary dimension vector. For a
coweight $\check\eta$, we use
$\op{IHom}(\check\eta):=\Hom_J(I([\check\eta]_+),I([-\check\eta]_+))$,
and write $\dv(\check\eta)$ for the ordinary dimension vector of a
general representation obtained from this space.
The resulting irreducible components of decorated representation varieties
are called \emph{principal components}.
The principal-component theorem \cite[Theorem~3.11]{FeiR25} identifies the
projective and injective components when
$\check\eta=\eta+\dv(\eta)B$, and shows that $\tau$ preserves
general representations. At the level of presentations, it is induced by
the Nakayama functor
\[ \op{PHom}(\eta)\xrightarrow{\ \nu\ }\op{IHom}(-\eta). \]
Applying this isomorphism to products of presentation spaces also
preserves general pairs. We also write $\tau$ for the induced
translation of weights and coweights. Formula \eqref{eq:app-weights} gives
\begin{equation}\label{eq:generic-weight-AR} \tau^{-1}\eta=-\eta-\dv(\eta)B,\qquad \tau\check\eta=-\check\eta+\dv(\check\eta)B.\end{equation}

For a general representation of coweight $\check\eta$, we write
$f_{\check\eta}$ for its tropical $F$-polynomial
\cite[Section~6.1]{FeiT23}. Define the
\emph{tropical Fock--Goncharov pairing} and its Hom counterpart by
\[ (\gamma,\delta)_{\mathrm{FG}}=f_{\gamma B}(\delta)-\delta(\gamma),\qquad (\gamma,\delta)_{\mathrm{Hom}}=\hom(\delta,\gamma B)-\delta(\gamma). \]
Here $\gamma B$ is a coweight, and Hom dimension is taken on a general
pair in the indicated principal components. We do not assume equality
of the two pairings.

\begin{corollary}\label{cor:generic-pairing-AR}
For $\gamma\in\ZZ^{Q_0}$, put $\gamma'=\dv(\gamma B)-\gamma$.
Under the hypotheses of Proposition~\ref{prop:AR-transformations}(a), for every
$\delta\in\RR^{Q_0}$,
\begin{equation}\label{eq:generic-FG-AR} (\gamma',\mathcal D^{-1}(\delta))_{\mathrm{FG}}=(\gamma,\delta)_{\mathrm{FG}}.\end{equation}
With no hypothesis on $\mathcal D$, for every presentation weight
$\delta\in\ZZ^{Q_0}$,
\begin{equation}\label{eq:generic-Hom-AR} (\gamma',\tau\delta)_{\mathrm{Hom}}=(\gamma,\delta)_{\mathrm{Hom}}.\end{equation}
\end{corollary}

\begin{proof}
If $M$ is general of coweight $\gamma B$, then $\tau M$ is general
of coweight $\gamma'B$, by \eqref{eq:generic-weight-AR}.
Evaluating \eqref{eq:invariant-pairing} at $\mathcal D^{-1}(\delta)$
gives \eqref{eq:generic-FG-AR}.

For the Hom identity, choose $L,M$ as a general pair, with $L$
of weight $\delta$. We subtract
$(\tau\delta)(\gamma')-\delta(\gamma)$ from \eqref{eq:hom-AR}.
Using $\gamma'=\dv(\gamma B)-\gamma$, the difference of the two
pairings reduces to
\[ (\tau\delta+\delta)(\gamma)-(\gamma B)(\dv(\tau\delta))=0, \]
by \eqref{eq:generic-weight-AR} and skew-symmetry of $B$.
This proves \eqref{eq:generic-Hom-AR}.
\end{proof}

We distinguish the tropical coordinate transformation $\mathcal D$ from
$\tau^{-1}$, the translation of generic presentation weights. Their
defining formulas are
\[ \mathcal D(\eta)=-\eta-\mathbf f_I(\eta)B,\qquad \tau^{-1}\eta=-\eta-\dv(\eta)B. \]
Thus they agree at every weight $\eta$ for which
$\mathbf f_I(\eta)=\dv(\eta)$. The preceding pairing identities do
not require this additional identification.

\begin{example}[A nonrigid family of the Markov quiver]
Put $s=\delta_1+\delta_2+\delta_3$. From
\eqref{eq:markov-A}--\eqref{eq:markov-F}, we obtain
\[ f_{A_i}(\delta)=\max\{0,\delta_i,\delta_i+2\delta_{i-1},2s-\delta_i,2s\},\qquad f_{I_i}(\delta)=2[s]_++f_{A_i}(\delta). \]
The common term cancels, so
\begin{equation}\label{eq:markov-tropical-y} \mathcal D(\delta)(i)=-\delta_i+2f_{A_{i+1}}(\delta)-2f_{A_{i-1}}(\delta).\end{equation}
By \eqref{eq:markov-psi}, this is an integral piecewise-linear involution
sending $s$ to $-s$. It fixes $s=0$ pointwise: substitute
$z_3=(z_1z_2)^{-1}$ in $\psi|_{q=1}=\mathrm{id}$ and tropicalize.
The compatibility in Section~\ref{subsec:markov-characters} becomes
\[ D(\omega)B^{\mathsf t}=\mathcal D(\omega B^{\mathsf t}). \]
The monomial map tropicalizes to $\omega\mapsto\omega B^{\mathsf t}$,
whose image is $s=0$. Since $\mathbf1 B^{\mathsf t}=0$, the
nontrivial displacement $D(\omega)-\omega=2r(\omega)\mathbf1$
disappears in these coordinates.

For $\lambda\in\mathbb C^\times$, let $R_\lambda$ have dimension
vector $(1,1,0)$, with $a_1=1$, $a_2=\lambda$, and all other
arrows zero. The projectives described in \cite[Example~4.3]{Pla13} give
\[ 0\longrightarrow R_\lambda\longrightarrow P_2\xrightarrow{\,a_2-\lambda a_1\,}P_1\longrightarrow R_\lambda\longrightarrow0. \]
To check the kernel, let $u_5^{(2)},v_5^{(2)}$ be the alternating
paths of length five from vertex $2$, starting with $b_1,b_2$,
and let $s_2$ be the common socle path. The kernel has basis
$\lambda u_5^{(2)}+v_5^{(2)}$ at vertex $1$ and $s_2$ at
vertex $2$, with arrow maps $1,\lambda$. The symmetrizing form
in Section~\ref{subsec:markov-characters} identifies the Nakayama
transform with the same multiplication map. Hence
\[ \tau R_\lambda\simeq R_\lambda,\qquad \check\delta_{R_\lambda}=(-1,1,0). \]
Self-extensions are supported on vertices $1,2$, where the relations
impose no conditions. The usual two-arrow calculation gives
$\dim\Ext^1(R_\lambda,R_\lambda)=1$, so this family is nonrigid.

The three submodule dimension vectors are $0,(0,1,0),(1,1,0)$. Thus
\[ F_{R_\lambda}(z)=1+z_2+z_1z_2,\qquad \CC_T(R_\lambda;x)=\kappa x_3. \]
Here the character is that of a lift to the cluster category. In the
coordinates \eqref{eq:markov-ratio-coordinates}, we have
\[ \CC_T(R_\lambda;x)=\frac{u^2+v^2+1}{uv}. \]
It is independent of $h$, so \eqref{eq:markov-inversion} fixes this
character, as required by $\tau R_\lambda\simeq R_\lambda$.
Before specialization, the AR identity is
\[ 1+\psi_2+\psi_1\psi_2=\frac{A_2}{z_1z_2A_1}(1+z_2+z_1z_2). \]
We verify it by substitution from \eqref{eq:markov-A}. With
$f_{R_\lambda}(\delta)=\max(0,\delta_2,\delta_1+\delta_2)$,
its tropical form is
\[ f_{R_\lambda}(\mathcal D(\delta))=f_{R_\lambda}(\delta)-\delta_1-\delta_2-f_{A_1}(\delta)+f_{A_2}(\delta). \]
Together with \eqref{eq:markov-tropical-y}, this verifies the tropical
translation formula for the nonrigid family $R_\lambda$.
\end{example}

\appendix
\section{Conventions and disconnected positroids}\label{app:positroid}\label{sec:gluing}

For a source Postnikov cluster-tilting object $T=\bigoplus_iL_i$,
with $x_i=\Phi_T(L_i)$, we prove the initial-coordinate identity
\[ \tw_L^*(x_i)=\frac{\Phi_T(\Omega L_i)}{\Phi_T(P_{L_i})} \]
for disconnected positroids. Corollary~\ref{cor:seed} then gives the
formula on every character. We first compare the circle conventions,
and then prove that the initial identity is preserved under gluing.

\subsection{The circle conventions}
Our source convention in Section~\ref{sec:geometric} has
$\mathsf y^k=\mathsf x^{n-k}$, with $\mathsf x_i=t$ on $L_I$ for $i\in I$.
In \cite{JKS24,JRS26}, the relation is $\mathsf x^k=\mathsf y^{n-k}$, and
membership in the label means that $\mathsf x_i$ acts by $1$.
Let $r(i)=n+1-i$. Reflection of the circle gives a covariant exact
equivalence between these conventions:
\[ \mathsf x_i\longmapsto\mathsf y_{r(i)},\qquad \mathsf y_i\longmapsto\mathsf x_{r(i)},\qquad L_I\longmapsto M_{r(I)}. \]
For arrows indexed from $i-1$ to $i$, the vertex reflection is
$i\mapsto n-i$. The source necklace becomes the ordinary necklace
of the reflected positroid. The minor relabelling
$\Delta_I\mapsto\Delta_{r(I)}$ differs from column reflection by
$(-1)^{k(k-1)/2}$ in degree one and defines an isomorphism of
homogeneous coordinate rings. We transport the categorical results along
this equivalence, which preserves projectives; $\tw_L$ and $\tw_R$
continue to denote the twists in the original source convention.

We define $\mathcal B_\Pcal$ by transporting the necklace order of
$r(\Pcal)$ along this equivalence. Modules are Cohen--Macaulay over
$\mathbb C[[t]]$. We take a basic projective generator when the
necklace repeats. For disconnected diagrams, this definition uses the
necklace order rather than the unmodified dimer algebra.
We now prove the initial twist identities in this setting.
\subsection{A coordinate formula on the initial object}
Let $T=\bigoplus_{i\in\mathcal S}L_i$ be a rank-one cluster-tilting
object of $\GP\mathcal B_\Pcal$. We write $\Phi_T$ for the Laurent
expression of $\Phi_\Pcal$ in this cluster. Let $\mathsf F\subseteq\mathcal S$
index the projective-injective summands, and put
\[ A=\End(T)^{\mathrm{op}},\qquad e_{\mathsf F}=\sum_{i\in\mathsf F}e_i,\qquad \underline A=A/Ae_{\mathsf F}A,
\qquad H_i=\underline{\Hom}(T,L_i). \]
In this appendix we use left $A$-modules, equivalently right
$\End(T)$-modules, with the action induced by precomposition.
Thus $H_i$ is the corresponding projective left $\underline A$-module
for mutable $i$, and is zero for frozen $i$. Let
\[ \beta:K_0(\op{fd}\underline A)\longrightarrow K_0(\add T)\simeq\ZZ^{\mathcal S} \]
be the \emph{Euler-class map}, defined by Cohen--Macaulay approximations
as in \cite[Section~5.1]{JKS24}. We identify
$K_0(\CM A)\simeq K_0(\op{proj}A)\simeq K_0(\add T)$
using \cite[equation~(4.1)]{JRS26}. Here $\op{fd}$ denotes
finite-dimensional modules, and $K_0(\add T)$ is the \emph{split
Grothendieck group}, with $[L_i]$ identified with $\mathbf e_i$. We use the Fu--Keller character formula
\cite[Section~3]{FK10}, in the form
established for boundary orders in \cite[equation~(6.2)]{JKS24}
\begin{equation}\label{eq:FK} \Phi_T(M)=x^{\ind_T M} \sum_\gamma\chi_c\bigl(\Gr_\gamma\Ext^1(T,M)\bigr)x^{-\beta(\gamma)}.\end{equation}
Equivalently, its values on the mutable simples are the differences of
the two exchange middle terms with the sign in \eqref{eq:FK}; see
\cite[Proposition~5.9]{JKS24}.

\begin{lemma}\label{lem:projective-F}
With these conventions,
\begin{equation}\label{eq:projective-F} q_i(x):=\frac{\Phi_T(\Omega L_i)}{\Phi_T(P_{L_i})} =x_i^{-1}\sum_\gamma\chi_c\bigl(\Gr^\gamma H_i\bigr)x^{\beta(\gamma)}.\end{equation}
Here $\Gr^\gamma$ denotes the Grassmannian of quotient modules. Moreover,
$\sum_{j\in\mathcal S}\beta(\gamma)_j=0$.
\end{lemma}

\begin{proof}
We apply $\Hom(T,-)$ to a projective deflation
$P_{L_i}\twoheadrightarrow L_i$ and its kernel. This gives
\[ 0\to\Hom(T,\Omega L_i)\to\Hom(T,P_{L_i}) \to\Hom(T,L_i)\to H_i\to0. \]
Let $Z$ be the image of $\Hom(T,P_{L_i})\to\Hom(T,L_i)$.
It is Cohen--Macaulay, since it is a finitely generated torsion-free
$\mathbb C[[t]]$-module. Splitting the sequence at $Z$, we apply
\cite[Lemma~5.2]{JKS24} to the last short exact sequence and additivity in
$K_0(\CM A)$ to the first. We obtain
\[ \beta([H_i])=[\Hom(T,\Omega L_i)]-[\Hom(T,P_{L_i})]+[\Hom(T,L_i)]. \]
Under $K_0(\CM A)\simeq K_0(\add T)$, this becomes
\[ \ind_T(\Omega L_i)-[P_{L_i}]=-\mathbf e_i+\beta([H_i]). \]
Also $\Ext^1(T,\Omega L_i)\simeq H_i$. Substitution in \eqref{eq:FK}
and the correspondence between submodules of class $\alpha$ and quotients
of class $[H_i]-\alpha$ prove \eqref{eq:projective-F}.
Finally, the rank homomorphism sends every $[L_j]$ to $1$, whereas
finite-dimensional $A$-modules have rank zero. Applying it to
$\beta(\gamma)$ gives the asserted sum. This is also the initial-object
case of \cite[equations~(6.9) and (6.11)]{JKS24}.
\end{proof}

\subsection{Gluing along a frozen face}
Suppose a positroid splits across two complementary cyclic intervals,
with smaller positroids $\Pcal_1,\Pcal_2$. Choose source face collections
$\mathcal S_1,\mathcal S_2$, with distinguished frozen faces
$I_*\in\mathcal S_1$, $J_*\in\mathcal S_2$, so that a face collection
for the original positroid is
\begin{equation}\label{eq:faces} \mathcal S= \{I\cup J_*:I\in\mathcal S_1\} \cup\{I_*\cup J:J\in\mathcal S_2\}.\end{equation}
The common face $I_*\cup J_*$ is frozen. The ice quivers are glued at
this face, with no additional arrows involving mutable vertices.
We use \cite[Propositions~4.1 and 4.6]{Pre} for this face decomposition
and \cite[Section~6.4]{JRS26} for the corresponding necklace description.

We next compare the ideals of maps factoring through projective-injective
summands. An isomorphism of corner algebras alone would not suffice.

\begin{lemma}[Frozen factorization ideals]\label{lem:frozen-ideals}
Let $A$ be a basic algebra finite over $\mathcal R=\mathbb C[[t]]$,
complete in its radical topology, with primitive idempotents $e_v$
indexed by $V$. For $S\subseteq V$, write $e_S=\sum_{v\in S}e_v$.
Fix $\mathsf F\subseteq V$, and suppose
$V=V_1\cup V_2$ with $V_1\cap V_2\subseteq\mathsf F$.
Assume that every arrow of the Gabriel quiver having a mutable endpoint
has both endpoints in one of the sets $V_a$. Put
\[ \varepsilon_a=e_{V_a},\qquad \mathsf F_a=V_a\cap\mathsf F,\qquad A_a=\varepsilon_a A\varepsilon_a. \]
Then
\begin{equation}\label{eq:frozen-ideals} \varepsilon_a(Ae_{\mathsf F}A)\varepsilon_a=A_a e_{\mathsf F_a}A_a\qquad(a=1,2).\end{equation}
Moreover, the corner maps induce an isomorphism
\begin{equation}\label{eq:stable-product} A/Ae_{\mathsf F}A\simeq A_1/A_1e_{\mathsf F_1}A_1\times A_2/A_2e_{\mathsf F_2}A_2.\end{equation}
\end{lemma}

\begin{proof}
We use a completed path presentation of $A$. Such a presentation is
obtained by lifting a basis of each component of
$\op{rad}A/(\op{rad}A)^2$ and using completeness.
All the ideals in \eqref{eq:frozen-ideals} are closed: they are
$\mathcal R$-submodules of finite $\mathcal R$-modules, and the
radical and $t$-adic topologies are equivalent.

Consider a path with endpoints in $V_a$ that passes through a frozen
vertex. If the path stays in $V_a$, it factors through $e_{\mathsf F_a}$.
If it leaves $V_a$, take its first arrow leaving $V_a$.
Its initial vertex is frozen. Indeed, if that vertex were mutable,
it would belong to $V_a\setminus V_{3-a}$, and the arrow could not
have both endpoints in either $V_1$ or $V_2$.
The path therefore again factors through a vertex of $\mathsf F\cap V_a$.
The two factors have endpoints in $V_a$, so they are elements of
$A_a$, even if one makes an excursion outside $V_a$.
Passing to sums and limits proves
$\varepsilon_a(Ae_{\mathsf F}A)\varepsilon_a\subseteq A_a e_{\mathsf F_a}A_a$; the reverse inclusion is immediate.

Every path from $V_1\setminus\mathsf F$ to $V_2\setminus\mathsf F$, or in
the reverse direction, meets a frozen vertex by the same argument.
Thus the off-diagonal corners vanish in $A/Ae_{\mathsf F}A$. The images of
$\varepsilon_1,\varepsilon_2$ are orthogonal and sum to one because
$V_1\cap V_2\subseteq\mathsf F$. Formula \eqref{eq:frozen-ideals} identifies the two diagonal
corners and gives \eqref{eq:stable-product}.
\end{proof}

We apply the lemma to $A=\End(T)^{\mathrm{op}}$ and to the two
face collections obtained by adjoining $J_*$ and $I_*$, respectively,
as in \eqref{eq:faces}. These algebras are finite
$\mathbb C[[t]]$-orders. By \cite[Theorem~6.5]{JRS26}, their
Gabriel quivers agree with the plabic quivers after arrows between
frozen vertices have been omitted. Proposition~4.6 of \cite{Pre}
therefore verifies the arrow hypothesis of Lemma~\ref{lem:frozen-ideals}.
It is not necessary to control the omitted arrows between frozen vertices.

Adding the common subset to the rank-one labels identifies each corner
with the corresponding smaller endomorphism algebra, preserving its
vertex idempotents \cite[Corollary~6.4]{JRS26}. Under these
identifications, $e_{\mathsf F_a}$ is the smaller frozen idempotent. Thus
\eqref{eq:frozen-ideals} proves that the corner identification descends
to the stable quotients. Formula \eqref{eq:stable-product} becomes
\[ \underline A\simeq\underline A_1\times\underline A_2. \]
In particular, the projective $H_i$ for a mutable vertex on one side
is the corresponding projective for that factor; its quotient
Grassmannians agree with those in the smaller category.

\begin{lemma}\label{lem:beta-gluing}
Under this stable decomposition, the restriction of $\beta$ to either
factor is the exponent map of the corresponding smaller positroid, extended by zero to the other face
collection. The common frozen coordinate is retained.
\end{lemma}

\begin{proof}
It is enough to consider the classes of mutable simple modules, which
form a basis of the Grothendieck group of the finite-dimensional stable
algebra. For such a simple on the first side, the two exchange middle
terms involve only vertices in $V_1$, and are obtained from the exchange middle terms of the first collection by
adjoining $J_*$ to their labels. Their difference is
$\beta([S_i])$, with the sign fixed in \eqref{eq:FK}.
The agreement of these arrows, including those incident with frozen
vertices, follows from the same quiver identification and
\cite[Proposition~4.6]{Pre}. Thus $\beta([S_i])$ is the extension
by zero of $\beta_1([S_i])$. The second side is identical, and
additivity proves the assertion.
\end{proof}
We use variables $a_I$, $b_J$ for the two smaller clusters. The
\emph{Segre embedding} of homogeneous coordinate rings gives an injection
$\partial$ on the original function field with
\begin{equation}\label{eq:Segre} \partial(x_{I\cup J_*})=a_Ib_{J_*},\qquad \partial(x_{I_*\cup J})=a_{I_*}b_J.\end{equation}
For the moment use an ordering in which the two intervals are consecutive;
the signs for a general ordering are addressed below.

By Lemma~\ref{lem:beta-gluing}, a class $\gamma$ from the first stable
factor has exponent vector supported on the face collection obtained by adjoining $J_*$.
Since its coordinates sum to zero,
\begin{equation}\label{eq:beta-Segre} \partial(x^{\beta(\gamma)})=a^{\beta_1(\gamma)}.\end{equation}
Indeed, the possible power of $b_{J_*}$ is exactly
$\sum_j\beta(\gamma)_j=0$. Combining
\eqref{eq:projective-F}, \eqref{eq:stable-product}, and
\eqref{eq:beta-Segre} gives
\begin{equation}\label{eq:q-glue} \partial(q_{I\cup J_*})=q_I^{(1)}b_{J_*}^{-1},\qquad \partial(q_{I_*\cup J})=a_{I_*}^{-1}q_J^{(2)}.\end{equation}
For the shared frozen face this says $q_* =x_*^{-1}$.

The geometric left twist respects the direct-sum decomposition of the
matrices, hence the Segre decomposition of coordinate rings
\cite[Proposition~4.20 and Corollary~4.21]{Pre}. Assuming the initial
syzygy character identity for the smaller positroids, it follows that
\[ \partial\bigl(\tw_L^*x_{I\cup J_*}\bigr) =\tw_{L,1}^*(a_I)\,\tw_{L,2}^*(b_{J_*}) =q_I^{(1)}b_{J_*}^{-1}. \]
Together with \eqref{eq:q-glue} and injectivity of $\partial$, this
proves the initial identity on the first collection. The same argument applies to the second collection. Thus gluing preserves
the identity on the initial cluster.

For a general cut in the cyclic order, \eqref{eq:Segre} contains the
usual shuffle signs. Such a sign is a character of the column-weight
lattice. Every monomial $x^{\beta(\gamma)}$ has column weight zero: this
holds on the generators $\beta([S_j])$ because the two Pl\"ucker
exchange monomials have the same column weight. Thus no sign appears
in \eqref{eq:beta-Segre}. The sign attached to $x_i^{-1}$ in
\eqref{eq:projective-F} is the inverse of the sign attached to $x_i$,
and these are equal. The same computation therefore proves the result
with the standard signs as well.

Loops and coloops can be removed first. A \emph{loop} is a zero column;
a \emph{coloop} is a column belonging to every basis. Deleting such a position
from all rank-one profiles deletes a step common to every profile.
In the equations for a morphism between two profiles, that step repeats
the same valuation constraint. Thus deletion preserves the endomorphism
algebras, their stable projectives, and the exchange exponents used in
\eqref{eq:projective-F}. On the geometric side, a loop simply removes a
zero column. For a coloop, choose a representative with the coloop in a
one-dimensional summand, with entry $s\ne0$. Then
$\Delta_I=s\Delta_{I\setminus\{j\}}$, up to the fixed ordering sign,
and the normalized twist sends this summand to the one with entry
$s^{-1}$. Setting $s=1$ gives the degree-preserving minor
identification for the contracted positroid and intertwines the twists.
Consequently the initial identities are unchanged by either contraction.
A positroid with a single basis has just one homogeneous coordinate and
no mutable vertices; both transformations send that coordinate to its reciprocal.

Induction on connected components now proves the initial identities.
Corollary~\ref{cor:seed} gives \eqref{eq:positroid} for every object of
$\GP\mathcal B_\Pcal$, completing the proof of
Theorem~\ref{thm:geometric}.

\end{document}